\documentclass[12pt]{amsart}

\usepackage{amsmath} 
\usepackage{amsthm} 
\usepackage{amssymb} 

\usepackage{microtype} 
\usepackage{pinlabel} 

\newcommand{\sqdiamond}{\tikz [x=1.2ex,y=1.85ex,line width=.1ex,line join=round, yshift=-0.285ex] \draw  (0,.5) -- (0.75,1) -- (1.5,.5) -- (0.75,0) -- (0,.5) -- cycle;}%
\renewcommand{\Diamond}{$\sqdiamond$} 

\usepackage[scaled=0.9]{sourcecodepro} 

\usepackage{enumitem}
\usepackage{color}
\usepackage{subfig, caption}

\usepackage{wrapfig}
\usepackage{pdflscape}
\usepackage{array}
\usepackage{tikz-cd}
\usepackage{longtable, booktabs}

\usepackage[hidelinks, pagebackref]{hyperref}

\makeatletter
\define@key{href}{font}{#1}
\makeatother
\usepackage{xpatch}
\newcommand\hrefdefaultfont{\ttfamily}
\xpatchcmd\href{\setkeys{href}{#1}}{\setkeys{href}{font=\hrefdefaultfont,#1}}{}{\fail}

\renewcommand*{\backref}[1]{}
\renewcommand*{\backrefalt}[4]{
  \ifcase #1 
  [No citations.]
  \or [#2]
  \else [#2]
  \fi }

\usepackage{cite}
\makeatletter
\newcommand{\citecomment}[2][]{\citenum{#2}#1\citevar}
\newcommand{\citeone}[1]{\citecomment{#1}}
\newcommand{\citetwo}[2][]{\citecomment[,~#1]{#2}}
\newcommand{\citevar}{\@ifnextchar\bgroup{;~\citeone}{\@ifnextchar[{;~\citetwo}{]}}}
\newcommand{\citefirst}{\@ifnextchar\bgroup{\citeone}{\@ifnextchar[{\citetwo}{]}}}

\makeatother

\let\originalleft\left
\let\originalright\right
\renewcommand{\left}{\mathopen{}\mathclose\bgroup\originalleft}
\renewcommand{\right}{\aftergroup\egroup\originalright}

\newcommand{\calL}{\mathcal{L}}
\newcommand{\calM}{\mathcal{M}}
\newcommand{\calN}{\mathcal{N}}

\newcommand{\calS}{\mathcal{S}}
\newcommand{\calT}{\mathcal{T}}

\newcommand{\calV}{\mathcal{V}}

\newcommand{\HH}{\mathbb{H}}

\newcommand{\NN}{\mathbb{N}}

\newcommand{\RR}{\mathbb{R}}

\newcommand{\ZZ}{\mathbb{Z}}

\newcommand{\st}{\mathbin{\mid}} 
\newcommand{\from}{\colon} 
 
\newcommand{\Id}{\operatorname{Id}}

\newcommand{\cross}{\times}

\newcommand*{\medcap}{\mathbin{\scalebox{1.5}{\ensuremath{\cap}}}}
\newcommand*{\medcup}{\mathbin{\scalebox{1.5}{\ensuremath{\cup}}}}

\newcommand{\cover}[1]{{\widetilde{#1}}}

\newcommand{\closure}[1]{{\overline{#1}}}

\newcommand{\bdy}{\partial} 

\newcommand{\Isom}{\operatorname{Isom}} 

\newcommand{\Aut}{\operatorname{Aut}} 

\newcommand{\SL}{\operatorname{SL}} 

\input{header_article.tex}

\theoremstyle{plain}
\newtheorem{XXXtheoremQED}[equation]{Theorem} 
  {\pushQED{\qed}\begin{XXXtheoremQED}}
  {\popQED\end{XXXtheoremQED}}

\newcommand{\fakeenv}{} 

\newenvironment{restate}[2]  
{ 
 \renewcommand{\fakeenv}{#2} 
 \theoremstyle{plain} 
 \newtheorem*{\fakeenv}{#1~\ref{#2}} 
 \begin{\fakeenv}
}
{
 \end{\fakeenv}
}

\newenvironment{restated}[2]  
{ 
 \renewcommand{\fakeenv}{#2} 
 \theoremstyle{definition} 
 \newtheorem*{\fakeenv}{#1~\ref{#2}} 
 \begin{\fakeenv}
}
{
 \end{\fakeenv}
}

\usepackage{mathtools}

\usepackage{enumitem} 
\usepackage{MnSymbol} 
\usepackage{subfig, caption} 
\usepackage{wrapfig}
\usepackage{color}

\definecolor{mygray}{gray}{0.6}
\newcommand{\Loom}{\operatorname{\mathsf{Loom}}}
\newcommand{\Taut}{\operatorname{\mathsf{Taut}}}
\newcommand{\Veer}{\operatorname{\mathsf{Veer}}}
\newcommand{\stair}{\operatorname{\mathsf{\Gamma}}} 
\newcommand{\veer}{\operatorname{\mathsf{V}}} 
\newcommand{\loom}{\operatorname{\mathsf{L}}} 
\newcommand{\hull}{\operatorname{\mathsf{H}}}
\newcommand{\disk}{\operatorname{\mathsf{D}}}
\newcommand{\sector}{\operatorname{\mathsf{S}}}

\newcommand{\tetra}{\operatorname{\mathsf{tet}}}

\newcommand{\cell}{\operatorname{\mathsf{c}}}

\newcommand{\trace}{\operatorname{trace}}

\makeatletter
\newcommand{\dotDelta}{{\vphantom{\Delta}\mathpalette\d@tD@lta\relax}}
\newcommand{\d@tD@lta}[2]{%
  \ooalign{\hidewidth$\m@th#1\mkern-1mu\cdot$\hidewidth\cr$\m@th#1\Delta$\cr}%
}
\makeatother

\newcommand{\ExtDel}{\dotDelta}

\title{From loom spaces to veering triangulations}

\author[Schleimer]{Saul Schleimer}

\author[Segerman]{Henry Segerman} 

\date{\today}

\begin{document}

\begin{abstract}
We introduce \emph{loom spaces}, a generalisation of both the \emph{leaf spaces} associated to pseudo-Anosov flows and the \emph{link spaces} associated to veering triangulations.
Following work of Gu\'eritaud, we prove that there is a locally veering triangulation canonically associated to every loom space, and that the realisation of this triangulation is homeomorphic to $\RR^3$. 
\end{abstract}




\maketitle

\newlength{\savedintextsep}
\setlength{\savedintextsep}{\intextsep}

\setcounter{section}{1}
\setlength{\intextsep}{-3pt}
\begin{wrapfigure}[15]{r}{0.35\textwidth}
\vspace{8pt}
\scalebox{-1}[1]{\includegraphics[width=0.45\textwidth]{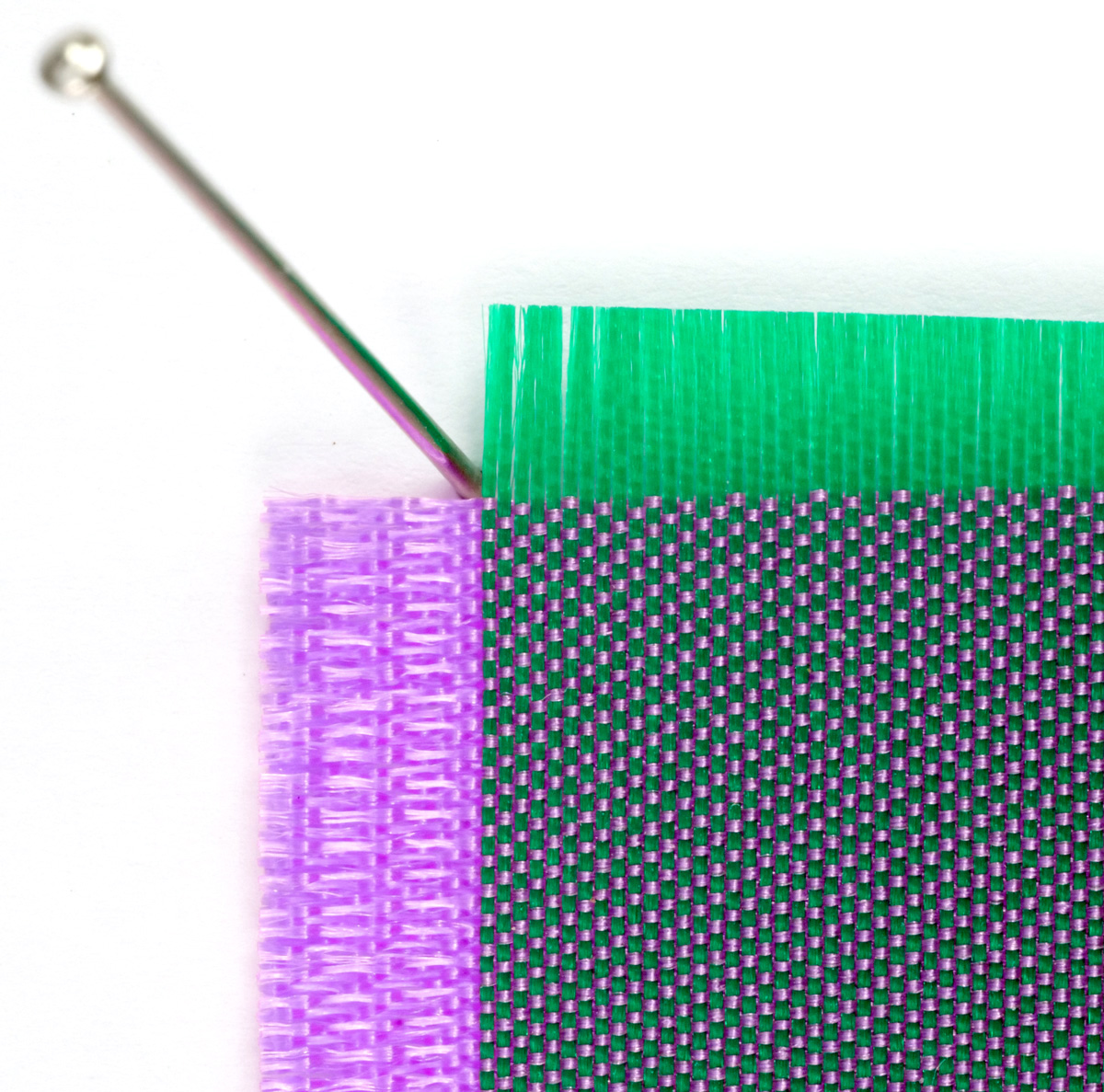}}
\caption{The green and purple threads give transverse foliations.}
\label{Fig:Fabric}
\end{wrapfigure}
\setcounter{section}{0}

\vspace{-15pt} 
\leavevmode
\section{Introduction}
\label{Sec:Intro}

\setlength{\intextsep}{\savedintextsep}

One-dimensional foliations, for example orbits of a flow, appeared early in the history of dynamical systems.
More delicate applications, such as foliations stable for, or transverse to, a flow, arrived in due course.
We refer to~\cite{HectorHirsch86} for a readable and well-illustrated introduction to this area.

Suppose that $S$ is a closed, connected, oriented surface.
Suppose that $f \from S \to S$ is a surface homeomorphism.
The \emph{mapping torus} for $f$ is the manifold $M(f)$ obtained from $S \cross [0, 1]$ by gluing, for every $x \in S$, the point $(x, 1)$ to the point $(f(x), 0)$.
Then $M(f)$ is equipped with a \emph{suspension flow} $\Phi(f)$ along the intervals; 
this flow has a transverse foliation given by the copies of $S$.
For an example in genus one, see \reffig{FigEightBox}.
Suspension flows associated to surface homeomorphisms are particularly important, for example due to the work of Thurston~\cite[Theorem~5.6]{Thurston82}.

When $f$ is (pseudo-)Anosov then we also have the \emph{stable} and \emph{unstable} (singular) foliations of $M(f)$ associated to $\Phi(f)$; 
the two-dimensional leaves of these foliations are obtained by taking suitable unions of flow-lines.
See Examples~\ref{Exa:AnosovMap} and~\ref{Exa:PseudoAnosovMap}; 
for more detail we refer to~\cite[Chapter~1]{Calegari07}. 
Unfortunately, the \emph{leaf space} of $\Phi(f)$ is highly non-Hausdorff.  
To obtain a somewhat calmer object, we define $\calL(f)$ to be the leaf space of the lift of $\Phi(f)$ to the universal cover of $M(f)$.
Since $S$ intersects every orbit of $\Phi(f)$, it follows that $\calL(f)$ is homeomorphic to $\cover{S}$, the universal cover of $S$.
Furthermore, the stable and unstable foliations for $\Phi(f)$ induce singular transverse foliations $F^f$ and $F_f$ of $\calL(f)$.

Following Agol and Gu\'eritaud~\cite{Agol11, Gueritaud16}, we form $S^\circ$ by removing the singular points from $S$.
Let $f^\circ$ be the restriction of $f$ to $S^\circ$.
Thus $M(f^\circ)$ is obtained from $M(f)$ by \emph{drilling}.  
The transverse foliations in $\calL(f^\circ)$ are now non-singular.

The goal of this paper is to generalise leaf spaces to what we call \emph{loom spaces}.
We provide axioms, discuss a number of examples, and draw out connections to \emph{veering triangulations}~\cite{Agol11}.

\subsection{This paper}

A \emph{loom space} $\calL$ is a copy of $\RR^2$ equipped with transverse (non-singular) foliations $F^\calL$ and $F_\calL$, satisfying two further axioms (\refdef{Loom}).  

In \refsec{LoomSpaces} we list several families of examples of loom spaces.  
We also discuss elementary relationships between the various \emph{skeletal rectangles} appearing in a loom space. 
In \refsec{Cusps} we formalise the notion of a \emph{cusp} of a loom space. 
These play an important combinatorial role in the rest of the work.

In \refsec{Astroid} we prove a key finiteness result: 
the \emph{astroid lemma} (\reflem{Astroid}).
This places a strong restriction on the projections of certain cusps to certain leaves of the two foliations of $\calL$. 
See \refrem{Others} for several versions of the astroid lemma appearing in previous work.

In \refsec{Veering} we review the basics of ideal triangulations and introduce \emph{locally veering triangulations}; 
we show in \refprop{LocallyVeering} that these are a mild generalisation of veering triangulations.  
In \refsec{Construction} we give our version of Gu\'eritaud's construction~\cite{Gueritaud16}.  
We then prove the following. 

\begin{restate}{Proposition}{Prop:Functorial}
Gu\'eritaud's construction is a functor from the category of loom spaces to the category of locally veering triangulations. 
\end{restate}


In \refsec{Convex} we define notions of \emph{geodesics} and \emph{convexity} in loom spaces.  
Using these we prove the following. 

\begin{restate}{Theorem}{Thm:ThreeSpace}
For any loom space, the topological realisation of its induced triangulation is homeomorphic to $\RR^3$. 
\end{restate}

\subsection{Related work}
\label{Sec:Related}

Let $\veer$ be the functor given by Gu\'eritaud's construction.
In joint work with Steven Frankel we show that $\veer$ is an \emph{equivalence} from $\Loom(\RR^2)$, the category of loom spaces, to $\Veer(\RR^3)$, the category of veering triangulations of $\RR^3$~\cite[Theorem~11.39]{FrankelSchleimerSegerman22}.
That is, there is another functor $\loom \from \Veer(\RR^3) \to \Loom(\RR^2)$ so that the $\loom \circ \veer$ and $\veer \circ \loom$ admit natural transformations to the identities on $\Loom(\RR^2)$ and $\Veer(\RR^3)$, respectively.
As a corollary, Gu\'eritaud's construction gives isomorphisms of automorphism groups. 
We prove the above by building, from a veering triangulation $\calV$ of $\RR^3$, a \emph{veering circle}, a pair of laminations in that circle, and thus a \emph{link space} $\loom(\calV)$.
After proving that $\loom(\calV)$ is a loom space we check naturality. 

In other work, joint with Jason Manning, we will show how the veering circle for $\calV$ compactifies the link space $\loom(\calV)$ to give the \emph{veering disk} $\disk(\calV)$. 
We will then use the astroid lemma (\reflem{Astroid}) to give a careful description of various Hausdorff limits in $\disk(\calV)$.  
Further work will prove, when $\calV$ gives a finite-volume cusped hyperbolic three-manifold $M$, that the \emph{veering two-sphere} is equivariantly homeomorphic to the Bowditch boundary of $\pi_1(M)$: that is, to $\bdy \HH^3$.
We will then use naturality to obtain new examples of Cannon-Thurston maps.

\subsection{Constructions of veering triangulations}

The definition and first construction of veering triangulations in the fibred case are due to Agol~\cite{Agol11}. 
The second author and collaborators generalised the definition~\cite{HRST11}; they also answered a question of Agol, using a computer search to find the first non-fibred examples. 
Gu\'eritaud~\cite{Gueritaud16} gave an alternative construction in the fibred case, which has inspired much later work, including this paper. We~\cite{Segerman15} announced a procedure to perform Dehn surgery along horizontal annuli or M\"obius strips in veering triangulations. 
We gave an implementation of a special case of this in the file \texttt{veering\_dehn\_surgery.py} in our codebase~\cite{Veering21}.
Shortly afterwards, Agol and Gu\'eritaud~\cite{Agol15} announced an extension of Gu\'eritaud's construction to drillings of manifolds admitting pseudo-Anosov flows without perfect fits.

A computer-generated census of all transverse veering triangulations with up to 16 tetrahedra was found by Giannopolous and ourselves~\cite{GSS19}.
Chi Cheuk Tsang~\cite{Tsang21} announced a procedure very similar to our veering Dehn surgery, which he calls \emph{horizontal surgery}.
He also introduced \emph{vertical surgery} along strictly ascending loops in the stable branched surface.
Landry, Minsky, and Taylor~\cite[Section~4]{LandryMinskyTaylor22} gave an exposition of the Agol-Gu\'eritaud construction. 
Moreover, they proved that the veering triangulation can be made smoothly transverse to the pseudo-Anosov flow.

\subsection*{Acknowledgements}

We thank the referee for their careful reading, which greatly improved the work. 
The second author was supported in part by National Science Foundation grants DMS-1708239 and DMS-2203993.
We thank Sabetta Matsumoto for sourcing the fabric shown in \reffig{Fabric}.
{We thank Sam Taylor for pointing out a mistake in our original definition of \emph{between};
this has now been fixed in \refdef{Between}.}

\section{Loom spaces}
\label{Sec:LoomSpaces}

\subsection{Rectangles}

Suppose that $\calL$ is a copy of $\RR^2$, equipped with two transverse foliations $F^\calL$ and $F_\calL$. 
We call these the \emph{upper} and \emph{lower} foliations respectively.


\begin{remark}
\label{Rem:Leaves}
The foliations $F^\calL$ and $F_\calL$ have no singularities in $\calL$. 
Thus, by the Poincar\'e--Hopf theorem~\cite[page~35]{Milnor65}, any two leaves are equal, disjoint, or intersect in exactly one point. 
We deduce that every leaf is properly embedded in $\calL$.
Thus, by the Jordan curve theorem~\cite[page~94]{Wilder49}, every leaf separates $\calL$. 
\end{remark}

\begin{definition}
\label{Def:Rectangle}
A \emph{rectangle} $R$ in $\calL$ is an open subset equipped with a homeomorphism $f_R \from (0,1)^2 \to R$.
We require that $f_R$ sends intervals parallel to the $x$--axis to arcs of $F_\calL$ and sends intervals parallel to the $y$--axis to arcs of $F^\calL$. 
\end{definition}

\begin{lemma}
\label{Lem:Basis}
With $\calL$ as above, the rectangles give a basis for the topology of $\calL$. \qed
\end{lemma}


\begin{remark}
\label{Rem:Cardinal}
Since $\calL$ is simply connected, we may choose orientations for the foliations $F^\calL$ and $F_\calL$.
When we do this, we also assume that all rectangle maps $f_R$ preserve these orientations.
This allows us to refer to the directions south, east, north, and west in $\calL$.
\end{remark}

\begin{definition}
\label{Def:Sides}
Suppose that $F^\calL$ and $F_\calL$ are oriented.
Suppose that $R$ is a rectangle in $\calL$. 
Let $\gamma_t \from (0,1) \to (0,1)^2$ be given by $\gamma_t(s) = (t,s)$.
The \emph{west side} of $R$ is the set of accumulation points of the sequence of arcs $(f_R(\gamma_t))_{t\to 0}$. 
We define the south, east, and north sides of $R$ similarly.
Intersections of sides, when they exist, are called \emph{material corners} of $R$.
\end{definition}

\begin{figure}[htbp]
\subfloat[Cusp rectangle]{
\label{Fig:CuspRect}
\includegraphics[width = 0.4\textwidth]{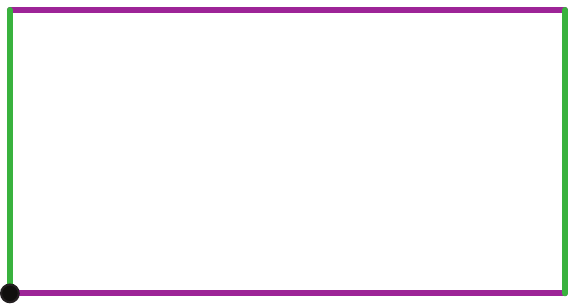}
}
\subfloat[Edge rectangle]{
\label{Fig:EdgeRect}
\includegraphics[width = 0.4\textwidth]{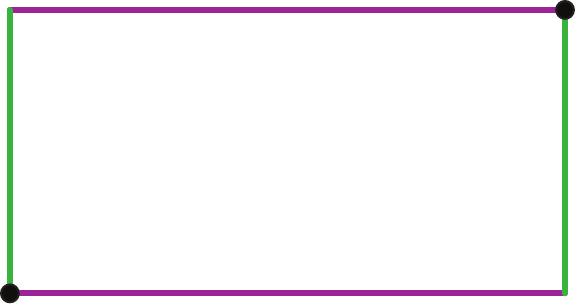}
}

\subfloat[Face rectangle]{
\label{Fig:FaceRect}
\includegraphics[width = 0.4\textwidth]{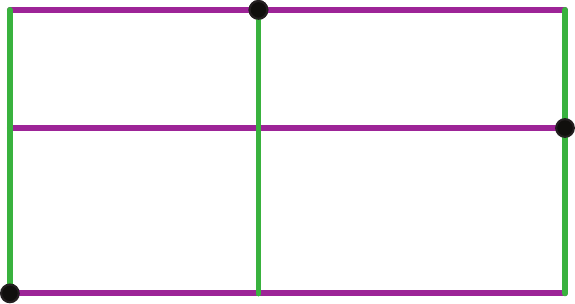}
}
\subfloat[Tetrahedron rectangle]{
\label{Fig:TetRect}
\includegraphics[width = 0.4\textwidth]{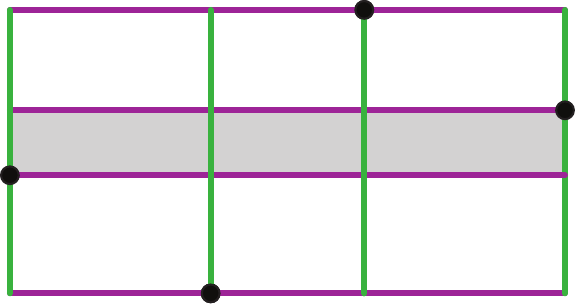}
}
\caption{Examples of cusp, edge, face, and tetrahedron rectangles. 
Here we indicate a point missing from the closure of a rectangle with a black dot.
In \refsec{Construction} we transform skeletal rectangles into cells of a triangulation.}
\label{Fig:SkeletalRects}
\end{figure}


\begin{definition}
\label{Def:CuspRect}
Suppose that $F^\calL$ and $F_\calL$ are oriented.
A rectangle $R$ in $\calL$ is a \emph{south-west cusp rectangle} if there is a continuous extension of $f_R$ to a homeomorphism 
\[
\closure{f}_R \from [0,1]^2 - \{(0,0)\} \to \closure{R}
\]
We define \emph{south-east}, \emph{north-east}, and \emph{north-west cusp rectangles} similarly. 
Suppose that $R$ is a south-west cusp rectangle. 
Note that the north and east sides of $R$ are closed intervals 
and that the south and west sides of $R$ are half-open intervals. 
We call the south and west sides of $R$ \emph{cusp sides}. 
We orient the south side to the east and the west side to the north.
We make similar definitions for the other types of cusp rectangle.
\end{definition}

See \reffig{CuspRect} for an example of a south-west cusp rectangle. 

\begin{definition}
\label{Def:TetRect}
A rectangle $R$ in $\calL$ is a \emph{tetrahedron rectangle} if there are $a, b, c, d \in (0,1)$ and a continuous extension of $f_R$ to a homeomorphism 
\[
\closure{f}_R \from [0,1]^2 - \{(a,0), (1,b), (c,1), (0,d)\} \to \closure{R}\qedhere
\] 
\end{definition}

See \reffig{TetRect} for an example of a tetrahedron rectangle.

\subsection{Loom spaces}

We are now ready to state our main definition.

\begin{definition}
\label{Def:Loom}
A \emph{loom space} $\calL$ is a copy of $\RR^2$ equipped with two transverse foliations $F^\calL$ and $F_\calL$ satisfying the following axioms.
\begin{enumerate}
\item 
\label{Itm:Cusp}
For every cusp side $s$ of every cusp rectangle, some initial open interval of $s$ is contained in some rectangle.
(See \reffig{CuspSide}.)
\item 
\label{Itm:Tet}
Every rectangle is contained in some tetrahedron rectangle. \qedhere
\end{enumerate}
\end{definition}




To explain the name \emph{loom space} we recall that the warp and weft of a fabric, as produced by a loom, give a pair of transverse foliations. 
See \reffig{Fabric}.

\begin{definition}
\label{Def:LoomIso}
Suppose that $\calL$ and $\calM$ are loom spaces.
We say that $f \from \calL \to \calM$ is a \emph{loom isomorphism} if
\begin{itemize}
\item
$f$ is a homeomorphism and 
\item
$f$ sends leaves to leaves. \qedhere
\end{itemize}
\end{definition}

\begin{wrapfigure}[11]{r}{0.24\textwidth}
\vspace{-12pt}
\centering
\labellist
\small\hair 2pt
 \pinlabel {$R$} at 82 68 
\endlabellist
\includegraphics[width = 0.22\textwidth]{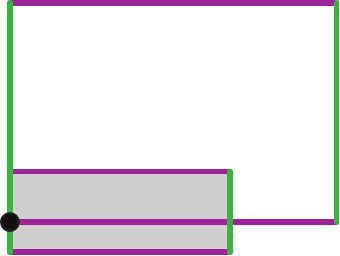}
\caption{A cusp rectangle $R$. An initial segment of its southern side is contained in the shaded rectangle.}
\label{Fig:CuspSide}
\end{wrapfigure}

Note that a loom isomorphism $f$ may send leaves of $F^\calL$ to leaves of $F^\calM$ \emph{or} to leaves of $F_\calM$.
We use $\Isom(\calL, \calM)$ to denote the set of loom isomorphisms from $\calL$ to $\calM$.
Note that loom isomorphisms compose in the usual way.
Thus loom spaces, together with their loom isomorphisms, form a category;
we denote this by $\Loom(\RR^2)$.
Finally, since loom isomorphisms have inverses 
the set $\Aut(\calL) = \Isom(\calL, \calL)$ is a group with respect to composition.



\subsection{Examples of loom spaces}

Our first loom space comes from a well-known example in dynamics.
The earliest exposition that we are aware of is due to Smale~\cite[page~757]{Smale67}.

\begin{example}
\label{Exa:AnosovMap}
Suppose that $A_0 \in \SL_2(\ZZ)$ is an \emph{Anosov matrix}: that is, $\trace(A_0)^2 > 4$.  
As an example, in \reffig{AnosovMap} we take
\[
A_0 = 
\begin{pmatrix}
2 & 1 \\
1 & 1
\end{pmatrix}
\]
Let $T = \RR^2 / \ZZ^2$ be the two-torus; 
let $A$ be the homeomorphism of $T$ induced by $A_0$. 
Let $F^A$ and $F_A$ be the resulting eigenfoliations in $T$.
Let $x \in T$ be the image of the origin.
Let $\calL$ be the universal cover of $T^\circ = T - \{x\}$.
Define $F^\calL$ and $F_\calL$ by lifting the eigenfoliations. 
We claim that $\calL$, equipped with these foliations, is a loom space.
To prove this, one uses any Markov partition of $T - \{x\}$, compatible with $A$, to verify the axioms of \refdef{Loom}.

The deck transformations of the covering give examples of loom isomorphisms. 
We obtain two more isomorphisms by lifting the actions of the matrices 
\[
R = 
\begin{pmatrix*}[r]
0 & -1 \\
1 &  0 
\end{pmatrix*}
\quad
\mbox{and}
\quad 
G = 
\begin{pmatrix}
1 & 1 \\
1 & 0
\end{pmatrix}
\]
on $\RR^2$ to $\calL$.  
It is an exercise to show that these (and the deck transformations) generate $\Aut(\calL)$.
\end{example}

\begin{figure}[htbp]
\labellist
\small\hair 2pt
 \pinlabel \rotatebox{31.718}{$S$} at 115 20
 \pinlabel \rotatebox{31.718}{$W$} at -25 63
 \pinlabel \rotatebox{31.718}{$N$} at 5 210
 \pinlabel \rotatebox{31.718}{$E$} at 150 170
\endlabellist
\includegraphics[width = 0.6\textwidth]{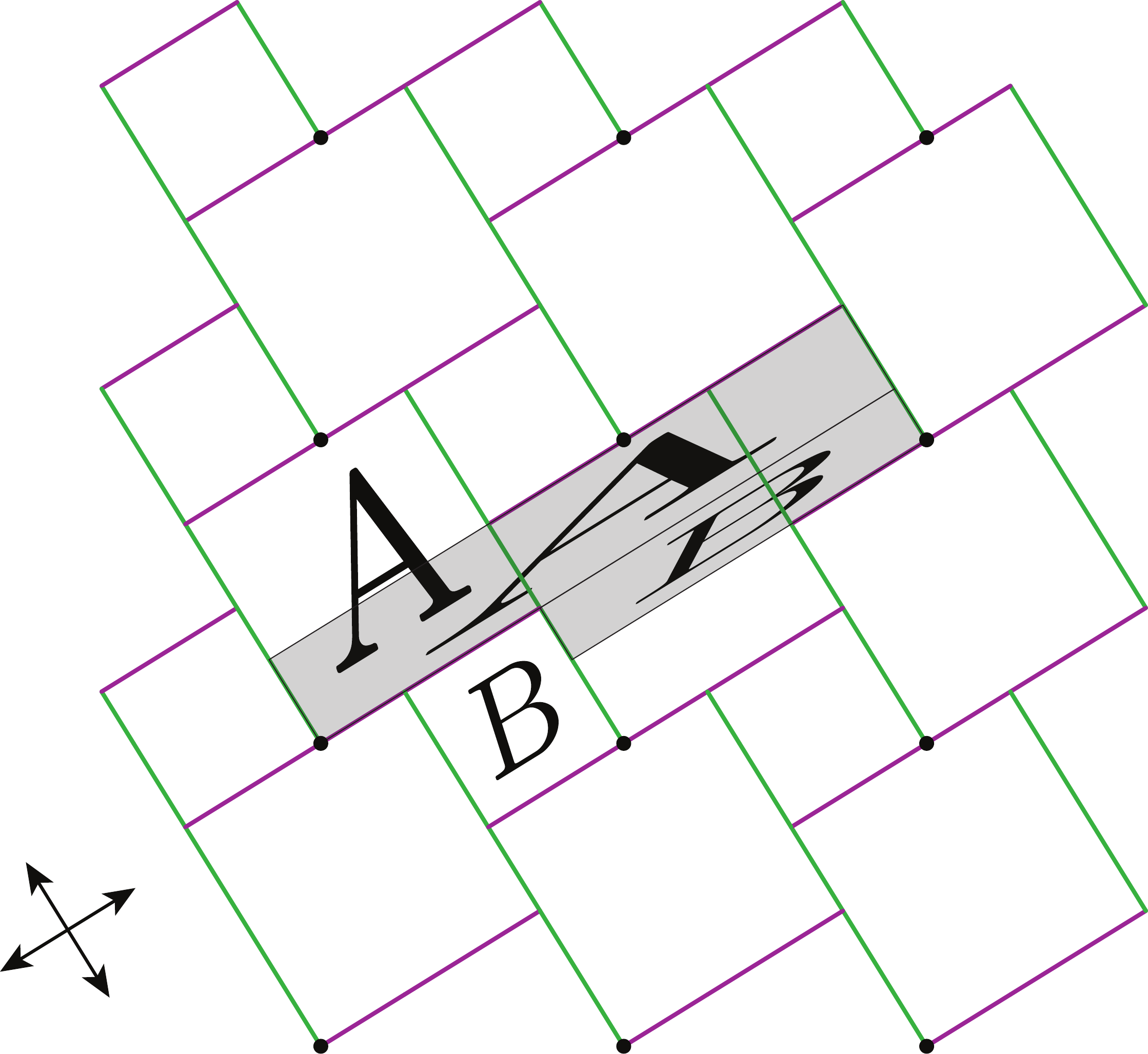}
\caption{The action of $A_0$ on its eigenfoliations.  
The dots are placed at integer lattice points. 
The rectangles containing letters (with their usual aspect ratio) are mapped by $A_0$ to the corresponding shaded rectangles.
These descend to $T$ to give a \emph{Markov partition}.}
\label{Fig:AnosovMap}
\end{figure}
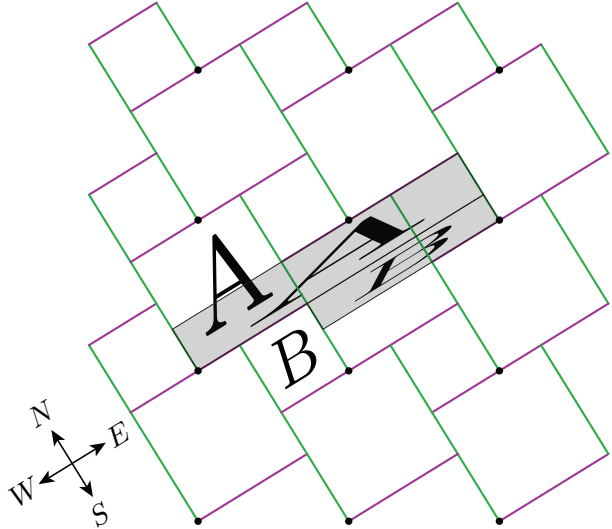

Our next family of examples comes from work of Thurston~\cite[Theorem~4(ii)]{Thurston88};
as their name indicates, these generalise \refexa{AnosovMap} to surfaces of higher genus.

\begin{example}
\label{Exa:PseudoAnosovMap}
Suppose that $S$ is a closed, connected, oriented surface with genus two or more. 
Suppose that $f \from S \to S$ is a \emph{pseudo-Anosov map}: 
that is, there are transverse measured singular foliations $F^f$ and $F_f$, each preserved leafwise by $f$, whose measures are, by $f$, respectively expanded and contracted by a common factor $\lambda_f > 1$. 
Let $Z \subset S$ be the set of singularities of $F^f$ and $F_f$.
Let $S^\circ = S - Z$.
Form $\calL$ by taking the universal cover of $S^\circ$ and lifting the foliations. 
We claim that $\calL$, equipped with these foliations, is a loom space.
As before, this is proved by building a Markov partition.

The element $f$ and the deck transformations of $\cover{S^\circ}$ generate a free-by-cyclic group.  
In fact, this group is a finite index subgroup of $\Aut(\calL)$~\cite[Theorem~11.39]{FrankelSchleimerSegerman22}.
For a discussion, see \refsec{Related}. 
\end{example}

We can generalise \refexa{PseudoAnosovMap} by instead taking $q$ to be a \emph{quadratic differential} on $S$.  
(Excellent introductions to abelian and quadratic differentials include~\cite{Zorich06} and~\cite{Wright15}.)
We must assume that the vertical and horizontal foliations $F^q$ and $F_q$ have no compact leaves.
Taking $Z$ to be the set of zeros of $q$, the rest of the construction is the same as \refexa{PseudoAnosovMap}.
This gives uncountably many examples of loom spaces.


Returning to the topological theme, suppose that $f \from S \to S$ is a surface homeomorphism.
As discussed in \refsec{Intro}, from $f$ we form the mapping torus $M(f)$ and its suspension flow $\Phi(f)$.
For an example, see \reffig{FigEightBox}.

Equipped with the results of this paper, it is a (difficult) exercise to show that $\Phi(f)$ is a \emph{pseudo-Anosov flow} if and only if $f$ is a pseudo-Anosov homeomorphism. 
For definitions, see~\cite[Section~6.6]{Calegari07}.
The exercise can also be deduced from various more general results in the literature: 
for example~\cite[Main~Theorem]{Fenley99b} or~\cite[Theorem~B]{Fenley03}.

Our next example generalises the exercise to other three-manifolds; 
we refer to Fenley's work, in particular~\cite[Definition~3.2]{Fenley98}, for an overview of pseudo-Anosov flows \emph{without perfect fits}.

\begin{figure}[htbp]
\includegraphics[width = 0.6\textwidth]{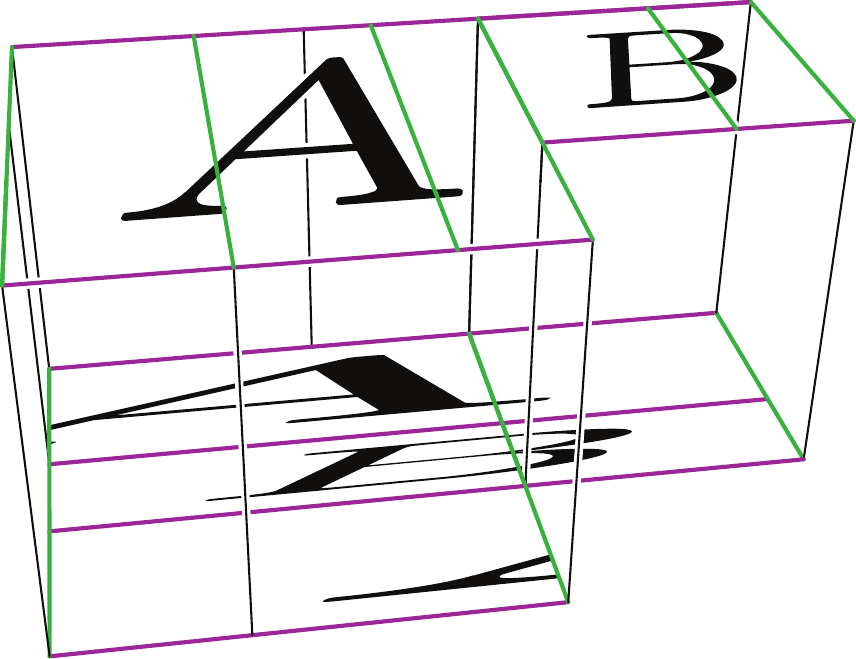}
\caption{A flow box decomposition of a torus bundle.
The associated flow is isomorphic to the suspension of \refexa{AnosovMap} and thus is Anosov.  
The manifold is homeomorphic to the longitudinal filling of the figure-eight knot complement.
Taking branched covers of the bundle, branched along the suspension of the origin, we obtain examples of pseudo-Anosov flows.}
\label{Fig:FigEightBox}
\end{figure}

\begin{example}
\label{Exa:PseudoAnosovFlow}
Suppose that $M$ is a closed, connected, oriented three-manifold.
Suppose that $\Phi \from M \cross \RR \to M$ is a \emph{topological pseudo-Anosov flow} (without perfect fits). 
(For the definition, see~\cite[page~80]{Mosher96}.)
Let $\Sigma^\Phi$ and $\Sigma_\Phi$ be the stable and unstable foliations of $M$. 
We remove from $M$ all singular flow loops to obtain the \emph{drilled space} $M^\circ$.  
We restrict $\Phi$ to $M^\circ$ to obtain $\Phi^\circ$ as well as the drilled foliations $\Sigma^\Phi_\circ$ and $\Sigma_\Phi^\circ$. 
We form the universal cover $\cover{M^\circ}$ and lift $\Phi^\circ$ as well as $\Sigma^\Phi_\circ$ and $\Sigma_\Phi^\circ$. 
The \emph{leaf space} $\calL^\circ = \calL\left(\cover{\Phi^\circ}\right)$ is the quotient of $\cover{M^\circ}$ by the flow $\cover{\Phi^\circ}$. 
The lifts of $\Sigma^\Phi_\circ$ and $\Sigma_\Phi^\circ$ descend to give non-singular foliations of $\calL^\circ$. 

Alternatively, we could form the universal cover $\cover{M}$ and the lifted flow $\cover{\Phi}$. 
Here the leaf space $\calL(\cover{\Phi})$ is homeomorphic to the plane $\RR^2$, equipped with a pair of singular foliations~\cite[Proposition~4.2]{FenleyMosher01}. 
See also~\cite[Lemma~6.53]{Calegari07}.
Furthermore, the singularities are a discrete subset of $\calL(\cover{\Phi})$.
So we may remove these and take a second universal cover.
This again gives $\calL^\circ$ and proves that it is homeomorphic to $\RR^2$.

In~\cite[Theorem~10.51]{FrankelSchleimerSegerman22} we give a combinatorial proof that $\calL^\circ$ is a loom space.
See also recent work of Landry, Minsky, and Taylor~\cite[Section~4]{LandryMinskyTaylor22}.
In addition, we show~\cite[Theorem~11.39]{FrankelSchleimerSegerman22} that $\pi_1(M^\circ)$ lies in $\Aut(\calL^\circ)$ as a finite index subgroup.
\end{example}

\begin{remark}
Pseudo-Anosov flows and maps are closely related, respectively, to
\emph{expansive flows} and maps.  These are defined by Bowen and
Walters~\cite{BowenWalters72}.  They give suspensions as a particular
example in Section~4 of~\cite{BowenWalters72}.
In their Theorem~6 they prove that $\Phi(f)$ is expansive if and only if $f$ is expansive. 
For further discussion of the subtle connections between expansive and pseudo-Anosov flows, we refer the reader to~\cite{Brunella95} and~\cite{BonattiWilkinson05}.
\end{remark}

\begin{wrapfigure}[18]{l}{0.27\textwidth}
\vspace{4pt}
\centering
\includegraphics[width = 0.25\textwidth]{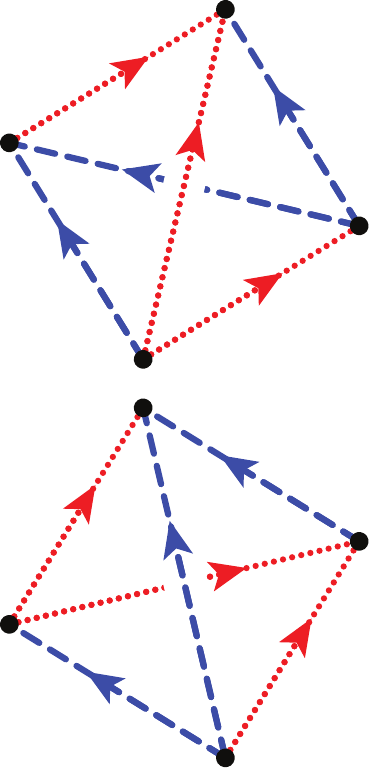}
\caption{The veering triangulation for the figure-eight knot complement.}
\label{Fig:FigEightTri}
\end{wrapfigure}

Removing the singular orbits of a pseudo-Anosov flow yields a manifold with torus boundary components. 
Our final examples are closely related to these \emph{drilled} flows but are completely combinatorial.
These rely on Agol's notion of a \emph{veering triangulation}~\cite[Definition~4.1]{Agol11}; 
see \refsec{Veering} for precise definitions and see \reffig{FigEightTri} for a concrete example.

\begin{example}
\label{Exa:Veering}
Suppose that $M$ is a compact, connected, oriented three-manifold with $\bdy M$ a non-empty collection of tori.
Suppose that $\calV$ is a \emph{veering} triangulation of $M$:
that is, an ideal triangulation of the interior of $M$ equipped with a \emph{taut} structure and a veering colouring.
(We give the definitions in \refsec{Definitions}.)
In~\cite[Theorems~10.51~and~11.39]{FrankelSchleimerSegerman22} we show that 
\begin{itemize}
\item
there is a canonical \emph{link space} $\calL$ associated to the lift of $\calV$ to the universal cover of $M$, 
\item 
$\calL$ is a loom space, and
\item
$\pi_1(M)$ is finite index in $\Aut(\calL)$. 
\end{itemize}
For a discussion see \refsec{Related}.
\end{example}

The overall goal of this paper is to provide the converse to \refexa{Veering}. 
In \refprop{Functorial}, from a given loom space, we build a \emph{locally veering} triangulation.  
In \refthm{ThreeSpace} we prove that the realisation of this triangulation is homeomorphic to $\RR^3$. 

\begin{remark}
\label{Rem:Aperiodic}
All of the examples of loom spaces given above have large automorphism groups.
It is interesting to contemplate how one might obtain a finitely described, yet aperiodic, loom space. 
\end{remark}




\subsection{Cusps repel}

From now on, we will assume that $\calL$, equipped with the foliations $F^\calL$ and $F_\calL$, 
is a loom space in the sense of \refdef{Loom}.  
We begin our analysis of $\calL$ by proving a version of a condition introduced by Keane~\cite[Section~2]{Keane75}.
In the language of foliations, Keane's \emph{minimality condition} on a singular foliation is that there are no leaves connecting singularities.
Our foliations $F^\calL$ and $F_\calL$ do not have singularities; 
we instead use the ``missing points'' from the sides of tetrahedron rectangles.

\begin{lemma}
\label{Lem:Keane}
Suppose that $R$ is a tetrahedron rectangle with associated parameters $a$, $b$, $c$, and $d$. 
Then $a \neq c$ and $b \neq d$.
\end{lemma}

\begin{wrapfigure}[5]{r}{0.24\textwidth}
\centering
\labellist
\scriptsize\hair 2pt
\pinlabel {$m$} [t] at 162 62 
\endlabellist
\includegraphics[width = 0.22\textwidth]{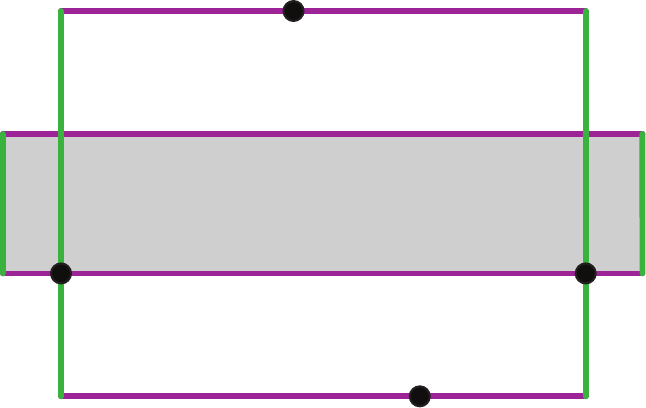}
\caption{}
\label{Fig:Keane}
\end{wrapfigure}

\begin{proof}
For a contradiction, suppose that $d = b$.
See \reffig{Keane}.
Let $m$ be the leaf of $F^\calL$ inside of $R$ running from $(0, b)$ to $(1, b)$. 
Fix a small positive number $\epsilon < 1/2$ so that $b + \epsilon < 1$. 
We now build two open rectangles in $\RR^2$ as products of open intervals. 
\[
A = (0, \epsilon) \cross (b, b + \epsilon)
\qquad
B = (1 - \epsilon, 1) \cross (b, b + \epsilon)
\]
Thus the images $f_R(A)$ and $f_R(B)$ are cusp rectangles in $\calL$ with south sides contained in the leaf $m$.

Applying \refdef{Loom}\refitm{Cusp} twice, we obtain small rectangles $A'$ and $B'$
containing initial segments of the west and east sides of $A$ and $B$, respectively.
Cutting these with a leaf of $F_\calL$ we obtain a single rectangle $C \subset \calL$ which has four material corners and which contains $m$ in the interior of its south side.
Thus $C$ is not contained in any tetrahedron rectangle.
This contradicts the axiom \refdef{Loom}\refitm{Tet}, as desired.

A similar argument deals with the case that $a=c$.
\end{proof}

\subsection{Skeletal rectangles}

For the next two definitions we choose orientations as in \refrem{Cardinal}.
See \reffig{SkeletalRects} for the following definitions and lemmas.

\begin{definition}
\label{Def:EdgeRect}
A rectangle $R$ in $\calL$ is a \emph{red edge rectangle} if there is a continuous extension of $f_R$ to a homeomorphism 
\[
\closure{f}_R \from [0,1]^2 - \{(0,0), (1,1)\} \to \closure{R}
\]
An edge rectangle $R$ is \emph{blue} if the missing points are instead $(0,1)$ and $(1,0)$.
\end{definition}

\begin{definition}
\label{Def:FaceRect}
A rectangle $R$ in $\calL$ is a \emph{south-west face rectangle} if there are $a, b \in (0,1)$ and a continuous extension of $f_R$ to a homeomorphism 
\[
\closure{f}_R \from [0,1]^2 - \{(0,0), (1,a), (b,1)\} \to \closure{R}
\]
We define the three other types of face rectangle similarly. 
\end{definition}

We call edge, face, and tetrahedron rectangles \emph{skeletal} rectangles. 
We do not include cusp rectangles among the skeletal rectangles because they are not uniquely determined by their cusps.

\begin{lemma}
\label{Lem:TetFaceEdge}
\leavevmode
\nopagebreak
\begin{itemize}
\item
Every tetrahedron rectangle contains exactly four face rectangles.
\item
Every tetrahedron rectangle contains exactly six edge rectangles.
\item
Every face rectangle contains exactly three edge rectangles.
\end{itemize}
\end{lemma}

\begin{proof}
Suppose that $R$ is a tetrahedron rectangle and let $f_R$ be the given parametrisation. 
There are at most four face rectangles in $R$; 
the sides of each intersect three of the four sides of $R$.

Let $(a, 0)$ and $(c, 1)$ be the missing points on the southern and northern sides.  
Appealing to \reflem{Keane} and breaking symmetry, suppose that $a < c$. 
Then 
\[
F = f_R \left( \{ (x, y) \in (0, 1)^2 \st x > a \} \right)
\]
is one of the desired face rectangles.  
The remaining three are formed similarly.

The other two statements are proved similarly.
\end{proof}

\begin{lemma}
\label{Lem:FaceTet}
Every face rectangle is contained in exactly two tetrahedron rectangles. 
\end{lemma}

\begin{proof}
Breaking symmetry, suppose that $F$ is a north-west face rectangle.  
See \reffig{TetPairingLoom}.
Let $\delta_F$ be the northern side of $F$. 
Suppose that $C$ is a small rectangle contained in $F$;
suppose that the north and west sides of $C$ are contained in the north and west sides of $F$. 
Let $\delta_C \subset \delta_F$ be the northern side of $C$. 

Since $C \subset F$ we deduce that $C$ is a cusp rectangle and that $\delta_C$ is a cusp side.   
By \refdef{Loom}\refitm{Cusp} we have that an initial segment of $\delta_C$ is contained in a rectangle, say $D$.
Note that $\epsilon = \delta_F - D$ is compact in $\calL$.  
Applying \reflem{Basis}, we cover $\epsilon$ by a finite collection of rectangles. 
We deduce that there is a rectangle $F'$ so that $F'$ contains both $F$ and $\delta_F$. 
We appeal to \refdef{Loom}\refitm{Tet} to obtain a tetrahedron rectangle $P$ containing $F'$.

Repeating the argument with the western side of $F$ gives another tetrahedron rectangle $Q$ containing $F$.
Note that $\delta_F$ lies in the interior of $P$ and lies in the northern side of $Q$. 
Hence $P$ and $Q$ are distinct.

We now show that there are at most two tetrahedron rectangles containing $F$.
Suppose that $R$ is any tetrahedron rectangle containing $F$. 
Then the southern and eastern sides of $R$ contain those of $F$.  
By \refdef{FaceRect}, the closure of $F$ is missing three points.  
The missing point at the north-western corner is necessarily the missing point from either the western or northern side of $R$.  
We deduce that $R$ is thus equal to $P$ or to $Q$. 
\end{proof}

It is more difficult to prove that an edge rectangle $E$ is contained in only finitely many face rectangles.  
This is deferred to \refcor{EdgeFace}.

\section{Cusps and corners}
\label{Sec:Cusps}

\begin{wrapfigure}[9]{r}{0.5\textwidth}
\vspace{-13pt}
\centering
\subfloat[]{
\labellist
\small\hair 2pt
 \pinlabel  $R_0$ at 177 25
 \pinlabel  $R_1$ at 137 107
 \pinlabel  $R_2$ at 50 141
\endlabellist
\centering
\includegraphics[width = 0.22\textwidth]{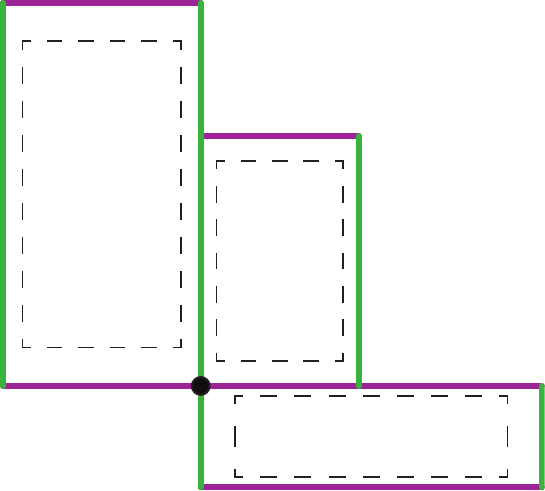}
}
\subfloat[]{
\labellist
\small\hair 2pt
 \pinlabel $R_0$ at 147 44
 \pinlabel $R_1$ at 67 122
\endlabellist
\centering
\includegraphics[width = 0.18\textwidth]{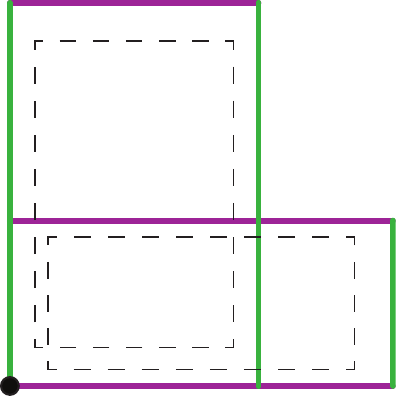}
}
\caption{}
\label{Fig:CuspRectangles}
\end{wrapfigure}

The \emph{cusps} of a loom space $\calL$ provide the beginnings of a boundary at infinity for $\calL$; 
we construct the boundary in future work, see \refsec{Related}.
This section provides the background needed for the statement of the astroid lemma. 

\begin{definition}
\label{Def:Equivalent}
Suppose that $R$ and $Q$ are cusp rectangles in $\calL$. 
We say that $R$ is \emph{equivalent} to $Q$ if there is a finite sequence of cusp rectangles 
\[
(R = R_0, R_1, \ldots, R_n = Q)
\]
so that for each pair $(R_i, R_{i+1})$ some cusp side of one is contained in some cusp side of the other.
\end{definition}

The equivalence relation is illustrated in \reffig{CuspRectangles}.

\begin{definition}
\label{Def:Cusp}
A \emph{cusp} is an equivalence class of cusp rectangles.
\end{definition}

For any cusp rectangle $R$, if $c = [R]$ then we refer to $R$ as a \emph{cusp rectangle for $c$}.

\begin{definition}
\label{Def:CuspsOf}
Suppose that $Q$ is a subset of $\calL$.
We say that $c$ is a \emph{cusp of $Q$}
if there is a sequence of rectangles $(R_i)_{i = 0}^\infty$ with the following properties:
\begin{enumerate}
\item
$R_i$ is a cusp rectangle for $c$, for all $i$.
\item
$\closure{R_i} \cap Q$ is non-empty, for all $i$. 
\item
$R_{i+1} \subset R_i$, for all $i$.
\item
The intersection $\medcap_i R_i$ is empty. 
\end{enumerate}
We define $\Delta(Q)$ to be the set of cusps of $Q$. 
\end{definition}

Recall from \refdef{Sides} that a rectangle may have as many as four material corners.
Cusps provide any remaining corners, as follows.

\begin{definition}
Suppose that $Q \subset \calL$ is a rectangle. 
Suppose that $Q$ contains a cusp rectangle $R$ where the cusp sides of $R$ are contained in sides of $Q$.
Then we call $c = [R]$ an \emph{ideal corner} of $Q$.
\end{definition}

\begin{lemma}
\label{Lem:Corners}
Every rectangle has four corners.
At most two of these are ideal.
\end{lemma}

\begin{proof}
Suppose that $R$ is the given rectangle.
By \refdef{Loom}\refitm{Tet} we have that $R$ is contained in a tetrahedron rectangle.
The result now follows from \reflem{Keane}.
\end{proof}

\begin{definition}
Fix a rectangle $R \subset \calL$. 
Suppose that $x$ and $y$ are corners (material or ideal) of $R$. 
We say that $x$ and $y$ are \emph{adjacent} if they are incident to a single side of $R$. 
If $x$ and $y$ are not adjacent then they are \emph{opposite}.
\end{definition}

\begin{definition}
Suppose that $\ell$ is a leaf of $F^\calL$ or $F_\calL$. 
Suppose that $R$ is a cusp rectangle in $\calL$ with a cusp side $\delta$. 
If $\delta$ is contained in $\ell$ then we call $\ell$ a \emph{cusp leaf for $c$}.
\end{definition}

It follows that $\delta$, the cusp side of $R$, contains an end of the leaf $\ell$. 

\begin{remark}
\label{Rem:CuspLeaves}
Suppose that $c$ is a cusp.  
As in \refrem{Leaves}, the Poincar\'e--Hopf theorem implies that any two cusp leaves for $c$ are disjoint. 
It follows that any leaf intersecting the interior of a rectangle $R$ is not a cusp leaf for an ideal corner of $R$.
\end{remark}

\begin{lemma}
\label{Lem:AtMostOne}
Any leaf $\ell$ of $F^\calL$ (or of $F_\calL$) is a cusp leaf for at most one cusp. 
\end{lemma}

\begin{proof}
Suppose that $c$ and $d$ are distinct cusps of $\calL$.
Suppose for a contradiction that $\ell$ is a cusp leaf for both $c$ and $d$.
Let $R$ and $Q$ be cusp rectangles at $c$ and $d$ with cusp sides $\gamma$ and $\delta$, both contained in $\ell$. 
If $\gamma$ and $\delta$ contain the same end of $\ell$ then 
$c = d$, contrary to assumption. 
Thus $\gamma$ and $\delta$ contain the two ends of $\ell$.
By \refdef{Loom}\refitm{Cusp} there are rectangles $R'$ and $Q'$ that contain initial segments of $\gamma$ and $\delta$, respectively, in their interiors.

Applying \reflem{Basis}, we cover the compact interval $\ell - (R' \cup Q')$ by finitely many rectangles.
We deduce that all of $\ell$ is contained in a single rectangle.
By \refdef{Loom}\refitm{Tet}, this rectangle is contained in a tetrahedron rectangle.
Appealing to \reflem{Keane}, we arrive at the desired contradiction.
\end{proof}

\begin{lemma}
\label{Lem:OneEdgeTwoCusps}
Suppose that $R$ is an edge rectangle.
Then $R$ has two cusps.
\end{lemma}

\begin{proof}
Let $x$ be an interior point of $R$. 
Let $\ell^x$ be the leaf of $F^\calL$ containing $x$. 
Let $P$ and $Q$ be the two components of $R - \ell^x$.
These are both cusp rectangles.  
We must prove that $[P] \neq [Q]$. 

Set $c = [P]$.
Let $\ell^c$ and $m_c$ be the cusp leaves containing the cusp sides of $P$.
Suppose that $P'$ is any cusp rectangle equivalent to, and disjoint from, $P$.
Let
\[
(P = P_0, P_1, \ldots, P_n = P')
\]
be a minimal sequence of cusp rectangles satisfying \refdef{Equivalent}.
Minimality implies that $P_i$ is disjoint from $P_{i+1}$. 
See \reffig{CuspRectangles}.
By \refrem{Leaves} one of the leaves $\ell^c$ or $m_c$ separates $P_1$ from $R$, and thus from $Q$.
By minimality and induction, the same leaf separates $P_k$ from $Q$ for all $k > 0$.
\end{proof}

\begin{lemma}
\label{Lem:TwoCuspsAtMostOneEdge}
Suppose that $R$ and $R'$ are edge rectangles. 
Then $R = R'$ if and only if $\Delta(R) = \Delta(R')$.
\end{lemma}

\begin{proof}
The forward direction follows from \refdef{CuspsOf}.
Now suppose that $R \neq R'$.
Breaking symmetry, suppose that $x$ lies in $R$ and does not lie in $R'$.
Let $\ell^x$ and $m_x$ be the leaves of $F^\calL$ and $F_\calL$ respectively that contain $x$.
By \refrem{Leaves}, the intersection $\ell^x \cup m_x$ is the singleton set $\{x\}$. 
By \refrem{CuspLeaves}, the leaves $\ell^x$ and $m_x$ are not cusp leaves for either cusp of $R$.
If $\ell_x$ and $m_x$ both intersect $R'$ then by \refdef{Rectangle} they meet in a point $y \in R'$.
Since $x \notin R'$ we have that $x \neq y$ and we have reached a contradiction.

Breaking symmetry, we have that $\ell_x$ does not intersect $R'$.
Let $A$ and $B$ be the two components of $\calL - \ell^x$.
Again appealing to \refrem{Leaves}, the rectangle $R'$ is in one of $A$ or $B$, but not both.
Using \refrem{Leaves} once more, the sets $\Delta(A) \cap \Delta(R)$ and $\Delta(B) \cap \Delta(R)$ are both singletons.
Therefore $\Delta(R) \neq \Delta(R')$.
\end{proof}

\section{The astroid lemma}
\label{Sec:Astroid}

Here we prove the \emph{astroid lemma} (\reflem{Astroid}). 
This controls the projection of certain cusps to certain leaves of $F^\calL$ and $F_\calL$. 

\subsection{Staircases}

Suppose that $x$ is a point or a cusp of $\calL$.
Fix any rectangle $R$ with a corner at $x$. 
Following Gu\'{e}ritaud~\cite[Section 4.3]{Gueritaud16}, we make the following definition.

\begin{definition}
\label{Def:Staircase}
The \emph{staircase} $\stair(x, R)$ is the closure of the union of all rectangles $Q \subset \calL$ where
\begin{itemize}
\item
$x$ is a corner of $Q$ and 
\item
$Q \cap R$ is non-empty. \qedhere
\end{itemize}
\end{definition}
\noindent
We often write $\stair(x) = \stair(x, R)$, suppressing the choice of $R$. 

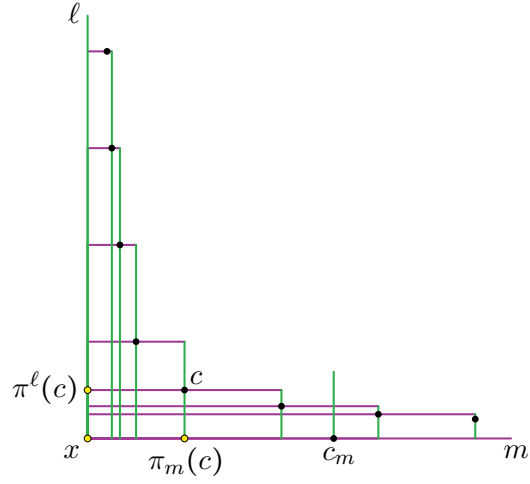
\begin{figure}[htbp]
\setlength{\abovecaptionskip}{20pt} 
\centering
\labellist
\small\hair 2pt
 \pinlabel  $x$ [tr] at 0 0
 \pinlabel  $c$ [bl] at 140 75
 \pinlabel  $\pi_m(c)$ [t] at 140 0
 \pinlabel  $\pi^\ell(c)$ [r] at 0 75
 \pinlabel  $\ell$ [br] at 0 570
 \pinlabel  $c_m$ [t] at 350 0
 \pinlabel  $m$ [tl] at 570 0
\endlabellist
\includegraphics[width=0.45\textwidth]{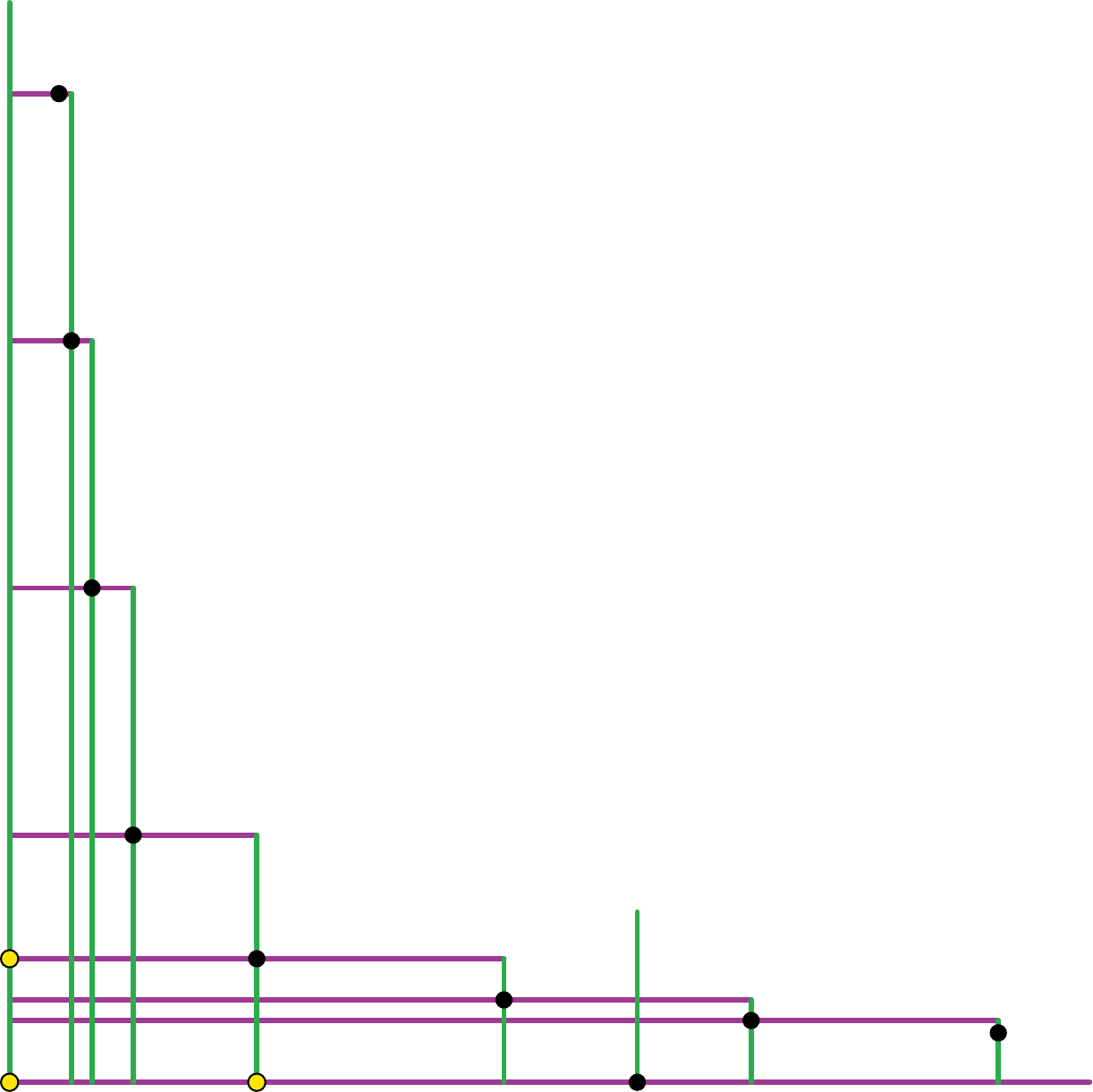}
\caption{A staircase. 
Cusps are indicated with black dots. 
Labelled material points are indicated by yellow dots.
In this example $x$ is a material point of $\calL$ to the west of a cusp $c_m$, the axis cusp of the lower axis ray $m$.}
\label{Fig:Staircase}
\end{figure}

\begin{definition}
\label{Def:AxisRays}
We take $m = m(x, R) \subset F_\calL$ to be the union of the arcs $s$ in $F_\calL$ so that there is a rectangle $Q$ so that
\begin{itemize}
\item
$Q \subset \stair(x, R)$, 
\item
$s$ is a side of $Q$, and
\item 
$s$ has an endpoint at $x$. 
\end{itemize}
We define $\ell = \ell(x, R) \subset F^\calL$ similarly.
We call $m$ and $\ell$ the \emph{lower} and \emph{upper axis rays}, respectively, for $\stair(x, R)$.
\end{definition}

\begin{definition}
\label{Def:AxisCusp}
Suppose that $\stair(x, R)$ is a staircase.
Suppose that $m$ and $\ell$ are its lower and upper axis rays.
Any cusp $c_m$ of $m$ is called an \emph{axis cusp} for $m$.
We make the same definition for a cusp $c_\ell$ of $\ell$.
\end{definition}

By \reflem{AtMostOne}, each axis ray has at most one axis cusp.

\begin{lemma}
\label{Lem:Initial}
Suppose that $\stair(x) = \stair(x, R)$ is a staircase. 
Breaking symmetry, suppose that $x$ is the south-west corner of $R$. 
Suppose that $m'$ is an initial segment of the lower axis ray $m$.
Then there is a rectangle $R' \subset \stair(x)$ so that 
\begin{itemize}
\item $x$ is the south-west corner of $R'$ and
\item $m'$ is the south side of $R'$.
\end{itemize}
\end{lemma}

\begin{proof}
Note that $x$ is either a point or a cusp of $\calL$. 
Thus using either \reflem{Basis} or \refdef{Loom}\refitm{Cusp}, respectively, there is an initial segment $m''$ of $m$ contained in a rectangle $Q$.
If $m''$ contains $m'$ then we cut $Q$ using the axis rays to obtain $R'$.

If not, there are two cases. 
Either the axis cusp $c_m$ is or is not contained in $m'$.
If $c_m$ is not contained in $m'$ then the interval $m' - Q$ is compact and so covered by a rectangle $Q'$.
Cutting $Q \cup Q'$ by the axis rays and reducing their height gives the desired rectangle $R'$.
If $c_m$ is contained in $m'$ then we apply \refdef{Loom}\refitm{Cusp} three times, and the remaining argument is as above.
\end{proof}

\begin{lemma}
\label{Lem:BdyQ}
Suppose that $x$ is the south-west corner of $R$. 
Then the lower axis ray $m$ has the following properties.
\begin{itemize}
\item
Suppose $x$ is a cusp.  
Then $m$ is a cusp leaf.
\item
Suppose $x$ is not a cusp.
Let $m_x$ be the leaf of $F_\calL$ containing $x$.
\begin{itemize}[label=$\circ$]
\item
If $m_x$ is a non-cusp leaf (or a cusp leaf emanating from the east of its cusp)
then $m$ is the eastern component of $m_x - x$.
\item
If $m_x$ is a cusp leaf emanating from the west of its cusp, then its cusp $c_m$ is the axis cusp for $m$. 
In this case $m - m_x$ is a cusp leaf emanating from the east of $c_m$.
\end{itemize}
\end{itemize}
A similar statement holds when $x$ is one of the other corners of $R$. 
Similar statements also hold for the upper axis ray $\ell$. 
\end{lemma}

\begin{proof}
Suppose that $x$ is a cusp.  
Shrinking $R$ slightly, we may assume that $R$ is a cusp rectangle for $x$.
Let $m'$ be the cusp leaf to the east of $x$ which contains the southern side of $R$.
We must show that the lower axis ray $m$ equals $m'$. 

Note $m$ is the union of connected sets (the southern sides of rectangles), 
all intersecting a connected set (the southern side of $R$).
Thus $m$ is connected. 
We deduce that $m$ is contained in $m'$. 

It remains to prove that $m'$ is contained in $m$. 
Let $m''$ be any closed initial segment of $m'$.
By \reflem{Initial}, the segment $m''$ is the southern side of some rectangle $R'$, showing that $m''$ lies in $m$.
We deduce that $m'$ is contained in $m$.

The remaining cases are similar.
\end{proof}

\begin{definition}
\label{Def:ExteriorCusps}
A cusp $c$ is an \emph{exterior cusp} of $\stair(x)$ if there is a rectangle in $\stair(x)$ having $c$ and $x$ as opposite corners.
We define $\ExtDel(\stair(x))$ to be the set of exterior cusps of $\stair(x)$.
\end{definition}

Note that $\ExtDel(\stair(x)) \subset \Delta(\stair(x))$.
When $x$ is a cusp, or when axis cusps exist, the containment $\ExtDel(\stair(x)) \subset \Delta(\stair(x))$ is proper. 
See \reffig{Staircase}. 

Let $\bdy m$ denote the end of the lower axis ray $m$ which is not at $x$ 
(or at the axis cusp $c_m$, if it exists).  
Note that $\bdy m$ is not at a cusp by Lemmas~\ref{Lem:AtMostOne} and~\ref{Lem:BdyQ}.
We define $\bdy \ell$ similarly.

We define a pair of projections 
\[
\pi_m \from \ExtDel(\stair(x)) \to m 
\quad 
\mbox{and} 
\quad
\pi^\ell \from \ExtDel(\stair(x)) \to \ell
\]
as follows.
Suppose that $c \in \ExtDel(\stair(x))$ is an exterior cusp.  
Then there is a rectangle $Q \subset \stair(x)$ with opposite corners at $x$ and $c$.  
We define $\pi_m(c)$ and $\pi^\ell(c)$ to be the corners of $Q$, 
\emph{other than} $x$, lying on $m$ and $\ell$ respectively.
See \reffig{Staircase}.

\subsection{Statement and proof}

We now control the images of the projections $\pi_m$ and $\pi^\ell$. 

\begin{lemma}[Astroid lemma]
\label{Lem:Astroid}
Suppose that $\stair(x)$ is a staircase in $\calL$; 
suppose that $m$ and $\ell$ are its axis rays. 
\begin{enumerate}
\item
\label{Itm:DoesNot}
The image of $\pi_m$ does not accumulate at any interior point of $m$ 
(nor does it accumulate at the axis cusp $c_m$, if present).  
\item
\label{Itm:Does}
The image of $\pi_m$ accumulates at $x$ and at $\bdy m$. 
\end{enumerate}
Similar statements hold for $\pi^\ell$. 
\end{lemma}

\begin{remark}
\label{Rem:Others}
Loom spaces associated to pseudo-Anosov homeomorphisms have various natural non-complete euclidean metrics.  
Each such metric has a definite injectivity radius;
the astroid lemma is immediate in these cases.
See~\cite[Lemma~14]{DelecroixUlcigrai15},~\cite[Figure~12]{Gueritaud16}, and~\cite[Figure~12]{Landry23}.

Loom spaces associated to pseudo-Anosov flows (without perfect fits) on finite volume hyperbolic three-manifolds need not have a natural choice of metric. 
However, in this setting the action of the fundamental group still gives a local finiteness that can replace the lower bound on injectivity radius.
This, in slightly different language, is carried out in~\cite[Section~4]{LandryMinskyTaylor22}. 
See in particular their Figure~18.
\end{remark}

\begin{proof}[Proof of \reflem{Astroid}]
Breaking symmetry, suppose that $x$ is south-west of $\stair(x)$.
Suppose that, in contradiction to \refitm{DoesNot}, 
there is a sequence of distinct exterior cusps $c_i \in \ExtDel(\stair(x))$ 
so that $r_i = \pi_m(c_i)$ accumulates at $r_\infty$, an interior point of $m$ (possibly the axis cusp $c_m$).
Define $s_i = \pi^\ell(c_i)$.  
Let $R_i \subset \stair(x)$ be the rectangle with corners at $x$, $r_i$, $c_i$, and $s_i$. 
Since the $c_i$ are all distinct, by \reflem{Keane} the points $r_i$ and $s_i$ are also all distinct.

We orient $m$ and $\ell$ away from $x$. 
We pass to a subsequence of the $c_i$ to ensure that the sequence $(r_i)$ is strictly monotonic in $m$.
Note that the rectangles $R_i$ cannot nest; 
we deduce that the sequence $(s_i)$ is strictly monotonic in $\ell$. 
Likewise, exactly one of the sequences $(r_i)$ and $(s_i)$ is increasing while the other is decreasing.
See \reffig{Astroid}.

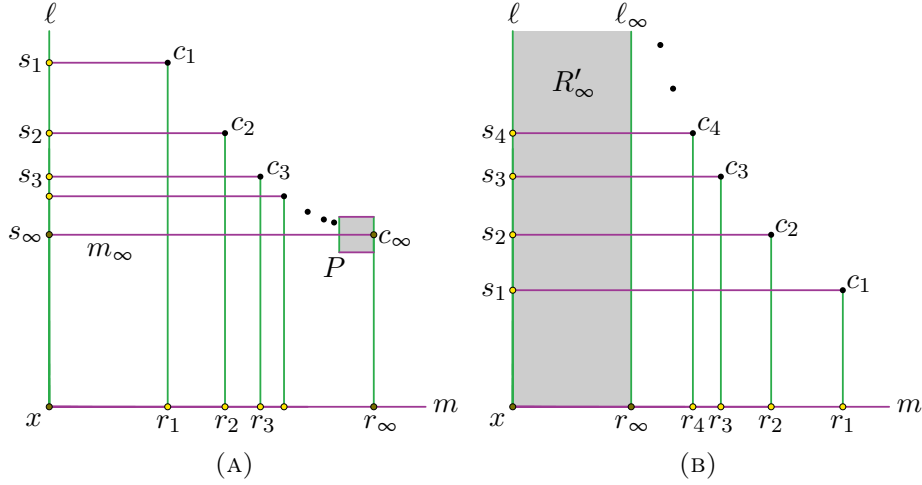
\begin{figure}[htbp]
\captionsetup[subfloat]{captionskip=15pt}
\subfloat[]{
\centering
\labellist
\small\hair 2pt
 \pinlabel  $x$ [tr] at 0 0
 \pinlabel  $\ell$ [b] at 7 590
 \pinlabel  $m$ [l] at 590 7
 \pinlabel  $c_1$ [bl] at 190 530
 \pinlabel  $r_1$ [t] at 190 0
 \pinlabel  $s_1$ [Br] at 5 530
 \pinlabel  $c_2$ [bl] at 280 420
 \pinlabel  $r_2$ [t] at 280 0
 \pinlabel  $s_2$ [Br] at 5 420
 \pinlabel  $c_3$ [bl] at 335 352
 \pinlabel  $r_3$ [t] at 335 0
 \pinlabel  $s_3$ [Br] at 5 352
 \pinlabel  $c_\infty$ [l] at 505 265
 \pinlabel  $r_\infty$ [t] at 518 0
 \pinlabel  $s_\infty$ [Br] at 5 265
 \pinlabel  $m_\infty$ [t] at 100 265
 \pinlabel  $P$ [t] at 445 245
\endlabellist
\includegraphics[width=0.4\textwidth]{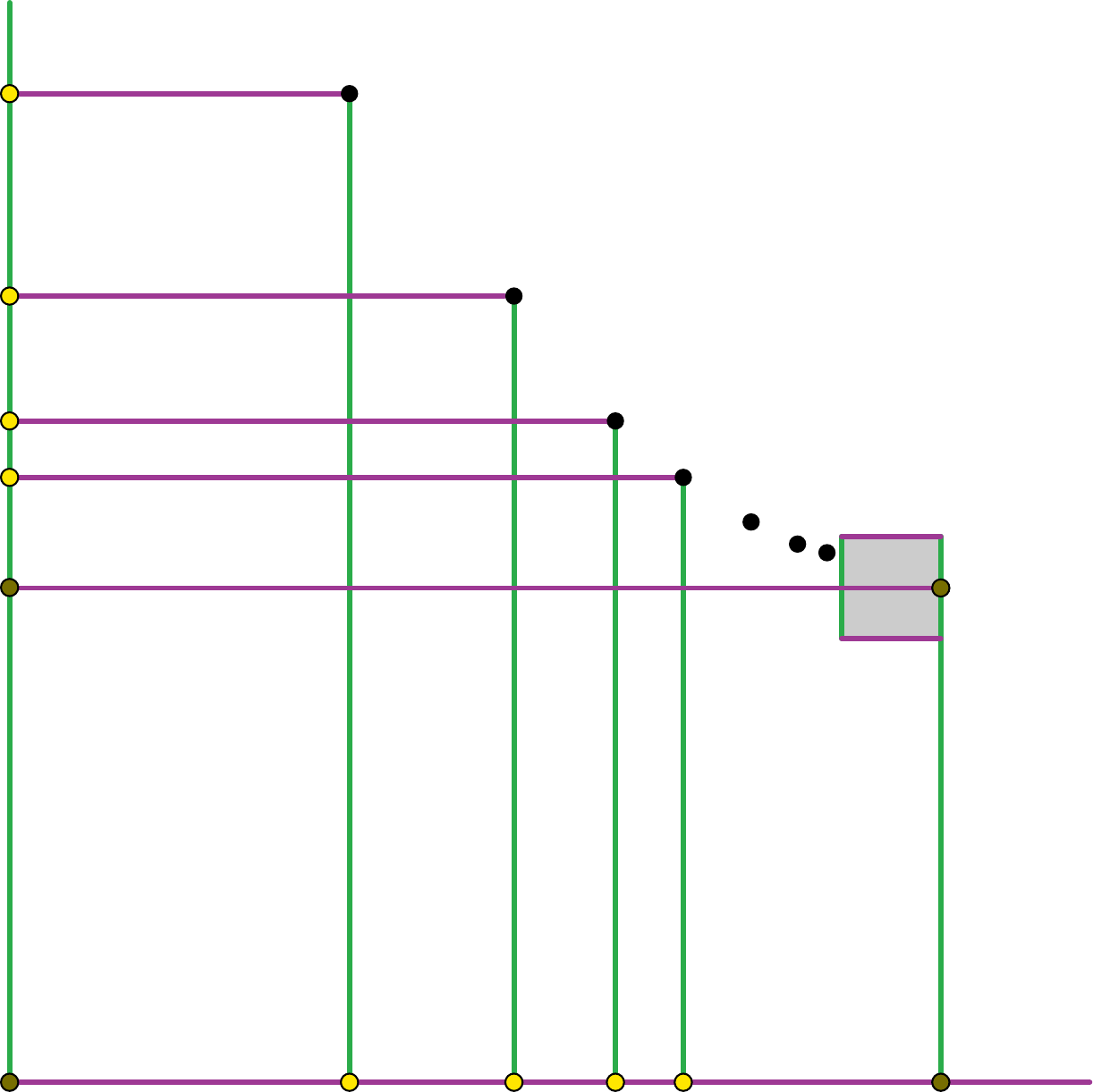}
\label{Fig:AstroidIncreasing}
}
\qquad
\subfloat[]{
\centering
\labellist
\small\hair 2pt
 \pinlabel  $x$ [tr] at 0 0
 \pinlabel  $\ell$ [b] at 7 590
 \pinlabel  $m$ [l] at 590 7
 \pinlabel  $\ell_\infty$ [b] at 190 583
 \pinlabel  $r_\infty$ [t] at 190 0
 \pinlabel  $c_4$ [bl] at 284 420
 \pinlabel  $r_4$ [t] at 284 0
 \pinlabel  $s_4$ [Br] at 5 420
 \pinlabel  $c_3$ [bl] at 327 353
 \pinlabel  $r_3$ [t] at 327 0
 \pinlabel  $s_3$ [Br] at 5 353
 \pinlabel  $c_2$ [bl] at 404 263
 \pinlabel  $r_2$ [t] at 404 0
 \pinlabel  $s_2$ [Br] at 5 263
 \pinlabel  $c_1$ [bl] at 515 177
 \pinlabel  $r_1$ [t] at 515 0
 \pinlabel  $s_1$ [Br] at 5 177
 \pinlabel  $R'_\infty$ at 100 505
\endlabellist
\includegraphics[width=0.4\textwidth]{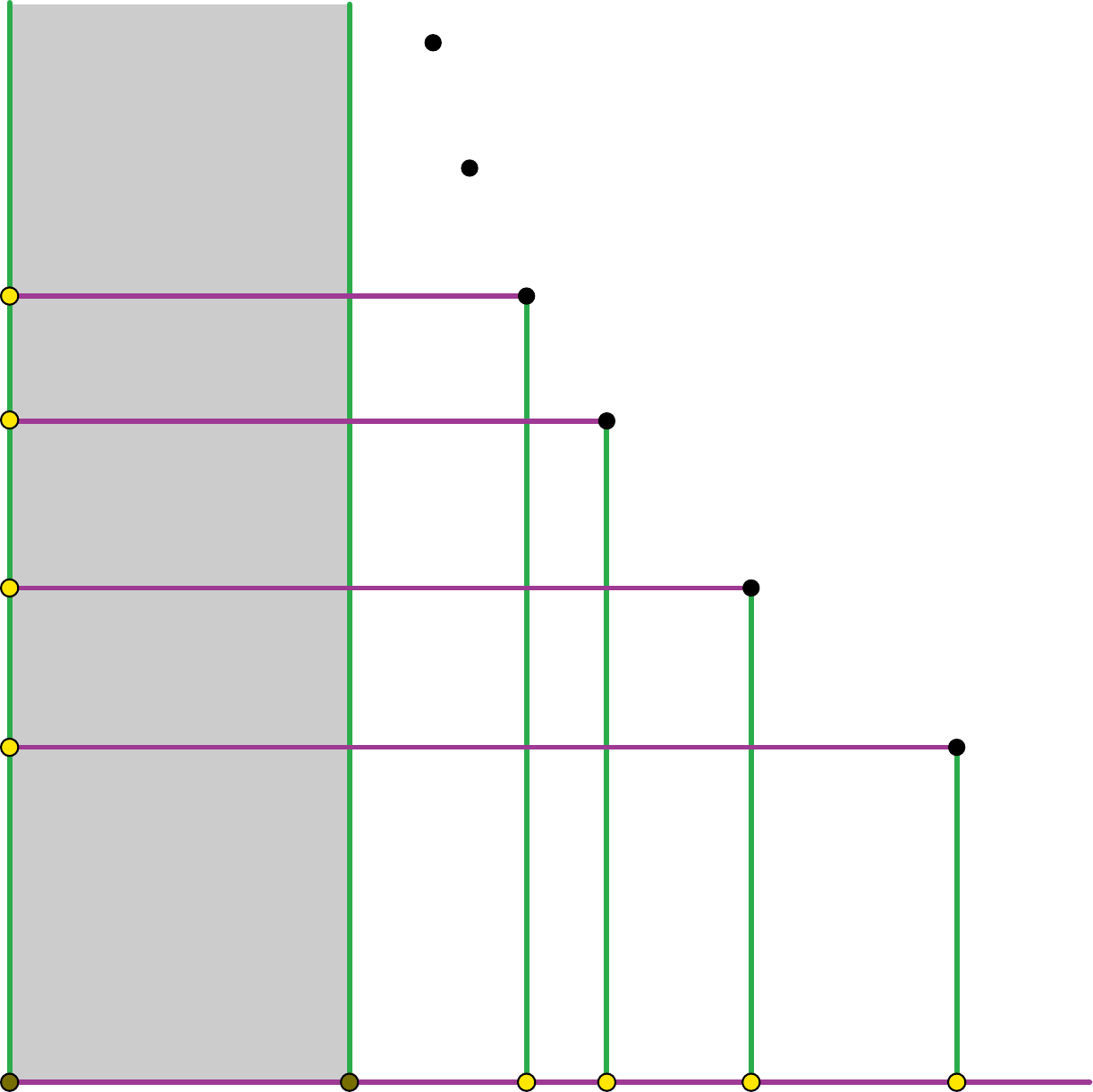}
\label{Fig:AstroidDecreasing}
}
\caption{The two possibilities as $r_i = \pi_m(c_i)$ is increasing or decreasing along $m$.}
\label{Fig:Astroid}
\end{figure}

We break the proof into two cases.

\begin{case*}
Suppose that $(r_i)$ is increasing in $m$.
\end{case*}

\noindent
Thus $(s_i)$ is decreasing along $\ell$.  
See \reffig{AstroidIncreasing}.
Since $m$ can be realised as an increasing union of southern sides of rectangles, by \reflem{BdyQ} we have a rectangle $Q$ in the staircase with corners at $x$ and $r_\infty$.  
Thus the points $s_i$ do not enter the interior of the western side of $Q$. 
That is, the sequence $(s_i)$ is bounded away from $x$ in $\ell$.
Thus there is some $s_\infty$ where they accumulate.  
Note that $s_\infty$ may be $c_\ell$, the axis cusp of $\ell$.   

Let $m_\infty$ be the ray of $F_\calL$ emanating from $s_\infty$ and entering $\stair(x)$. 
Recall that $R_i$ is the rectangle with opposite corners at $x$ and $c_i$.
Define $R'_i$ to be the component of $R_i - m_\infty$ with a corner at $x$.
Define $R'_\infty$ to be the increasing union of the $R'_i$. 
Thus $R'_\infty$ is a rectangle.  
Let $c_\infty$ be its northeastern corner.  

By \reflem{Corners} we have that $c_\infty$ is a point or a cusp of $\calL$.
Appealing to \reflem{Basis} or \refdef{Loom}\refitm{Cusp} respectively, there is a small rectangle $P$ so that
\begin{itemize}
\item
$c_\infty$ lies in the interior of the eastern side of $P$ and 
\item
the interior of $P$ intersects $m_\infty$.
\end{itemize}
However, the projections of $(c_i)$ accumulate on $s_\infty \in \ell$ and $r_\infty \in m$ respectively.
Thus the $c_i$ enter $P$, a contradiction.
Again, see \reffig{AstroidIncreasing}.

\begin{case*}
Suppose that $(r_i)$ is decreasing along $m$.
\end{case*}

\noindent
Thus $(s_i)$ is increasing along $\ell$.  
See \reffig{AstroidDecreasing}.
Let $\ell_\infty$ be the ray of $F^\calL$ emanating from $r_\infty$ and entering $\stair(x)$. 
Define $R'_i$ to be the component of $R_i - \ell_\infty$ with a corner at $x$. 
Define $R'_\infty$ to be the union of the $R'_i$. 
Again, $R'_\infty$ is a rectangle.
Note that the west side of $R'_\infty$ is contained in $\ell$. 
By \refdef{Loom}\refitm{Tet} there is a tetrahedron rectangle $Q$ containing $R'_\infty$.
So the west side of $R'_\infty$ is either contained in the interior of $Q$ or is contained in the west side of $Q$. 
In either case we obtain an upper bound for the $s_i$. 
We now apply the previous argument, swapping the roles of $m$ and $\ell$.  
This completes the proof of \refitm{DoesNot}.

To prove \refitm{Does} we must find a sequence of exterior cusps $c_i \in \ExtDel(\stair(x))$ whose projections $\pi_m(c_i)$ accumulate at $\bdy m$ and whose projections $\pi^\ell(c_i)$ accumulate at $x$. 
Let $(m_i)$ be an increasing sequence of open initial segments of $m$, whose union is $m$.   
By \reflem{Initial}, there is a rectangle $R_i$ with south-west corner at $x$ and whose southern side is $m_i$.  
Let $Q_i$ be the union of all rectangles $Q$ so that
\begin{itemize}
\item
$Q$ contains $R_i$ and
\item
the west sides of $Q$ and $R_i$ are identical.
\end{itemize}
Since $Q_i$ is a rectangle, by \refdef{Loom}\refitm{Tet} we have a tetrahedron rectangle, $R_i'$, containing $Q_i$.  
Thus there is a cusp $c_i$ contained in the east side of $R_i'$.
Note that $c_i$ is north of $m$, by the construction of $Q_i$.
So there is a rectangle in $\stair(x)$ with opposite corners at $x$ and $c_i$; 
thus $c_i$ is an exterior cusp of $\stair(x)$. 

By construction, the projection $\pi_m(c_i)$ is not contained in $m_i$. 
Since the $m_i$ exhaust $m$, the sequence of projections accumulates on $\bdy m$. 
It follows that the sequence of projections $\pi^\ell(c_i)$ is decreasing in $\ell$.
By \refitm{DoesNot}, the sequence $\pi^\ell(c_i)$ tends to $x$. 
This proves \refitm{Does} for $m$; 
the proof for $\ell$ is similar.
\end{proof}

We record a few consequences of the astroid lemma. 

\begin{corollary}
\label{Cor:CuspLeavesDense}
The cusp leaves are dense in $F^\calL$ and $F_\calL$. \qed
\end{corollary}


See also~\cite[Lemma~4.2(1)]{LandryMinskyTaylor22}.

\begin{corollary}
\label{Cor:EdgesOnStair}
Suppose that $c$ and $d$ are exterior cusps for $\stair(x)$.  
Suppose their projections to $\ell^x$ (or $m_x$) are consecutive.
(That is, not separated by the image of any other exterior cusp.)
Then there is an edge rectangle $R \subset \stair(x)$ having $c$ and $d$ as opposite corners. \qed 
\end{corollary}


\subsection{Finiteness and connectedness}

\begin{lemma}
\label{Lem:Finite}
For any rectangle $R$ there are only finitely many skeletal rectangles containing it.
\end{lemma}

\begin{proof}
Suppose that $R$ is the given rectangle.
Let $x$ be the south-west corner of $R$. 
Let $\stair(x) = \stair(x, R)$ the resulting staircase; 
let $m$ and $\ell$ be its axis rays. 

\begin{figure}[htbp]
\centering
\labellist
\small\hair 2pt
 \pinlabel  $x$ [tr] at 0 0
 \pinlabel  $R$ at 95 90
 \pinlabel  $\ell$ [b] at 7 590
 \pinlabel  $m$ [l] at 590 7
 \pinlabel  $c$ [b] at 150 555
 \pinlabel  $d$ [l] at 552 145
\endlabellist
\includegraphics[width=0.4\textwidth]{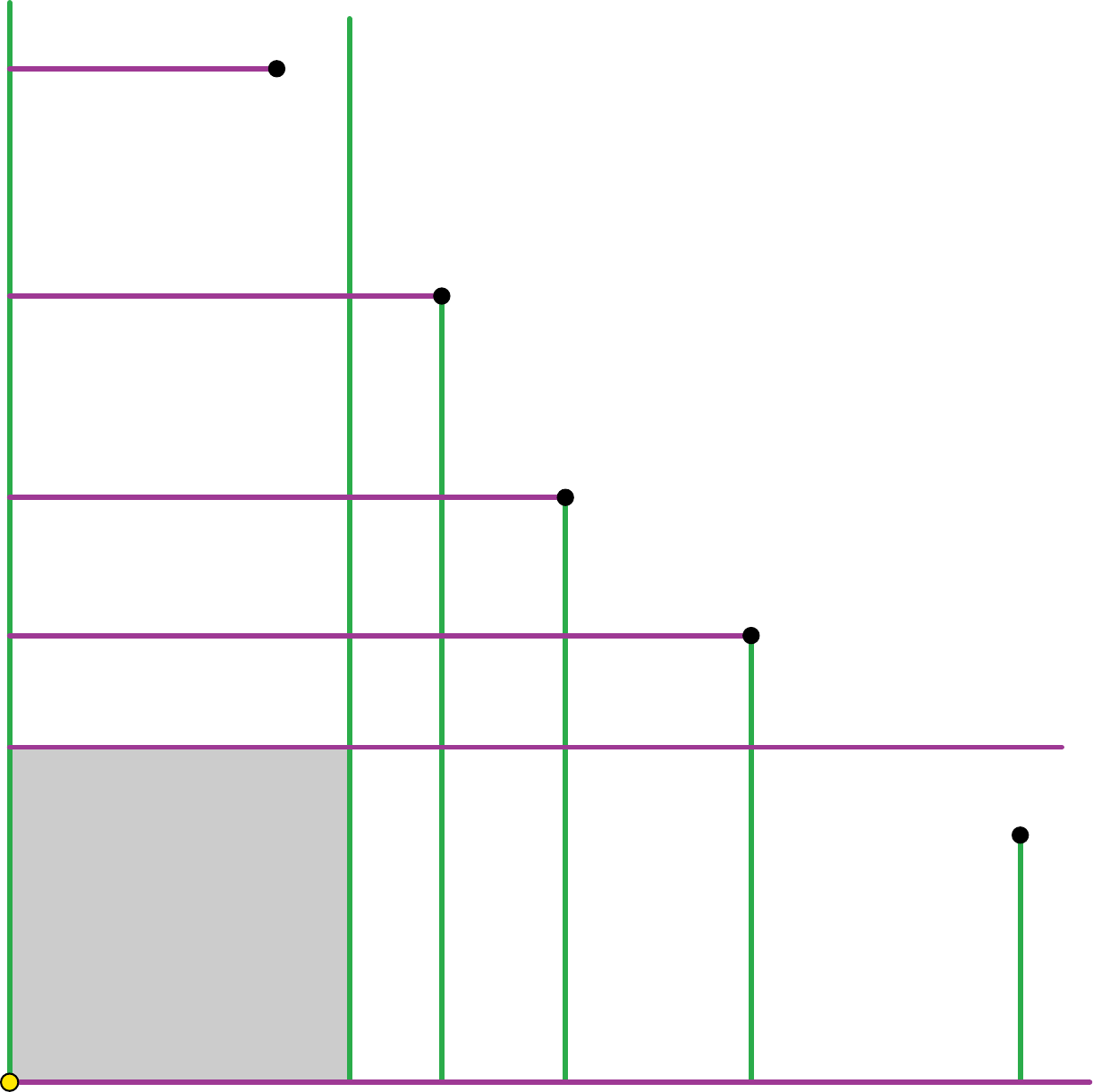}
\caption{The staircase $\stair(x, R)$.}
\label{Fig:MaximiseRectangle}
\end{figure}

Let $m_R \subset m$ be the projection of $R$ to $m$, along $F^\calL$. 
Similarly, let $\ell^R \subset \ell$ be the projection of $R$ to $\ell$, along $F_\calL$.
See \reffig{MaximiseRectangle}.
By \reflem{Astroid}\refitm{Does}, there are exterior cusps $c$ and $d$ in $\ExtDel(\stair(x))$ so that $\pi_m(c)$ lies in $m_R$ and $\pi^\ell(d)$ lies in $\ell^R$. 
By \reflem{Astroid}\refitm{DoesNot}, there are only finitely many cusps $c' \in \ExtDel(\stair(x))$ so that $\pi_m(c')$ lies between $\pi_m(c)$ and $\pi_m(d)$.
Furthermore, we may replace $x$ with any other corner of $R$ and perform the same analysis in the corresponding staircase. 
This determines a finite collection of cusps.
By \reflem{TwoCuspsAtMostOneEdge}, an edge rectangle is determined by its cusps.
A skeletal rectangle is determined by the edge rectangles it contains. 
This gives the desired bound. 
\end{proof}

\begin{corollary}
\label{Cor:EdgeFace}
Any edge rectangle is contained in only finitely many face rectangles.
\end{corollary}

\begin{proof}
This follows from Lemmas \ref{Lem:Finite} and \ref{Lem:TetFaceEdge}.
\end{proof}

Before giving the next result we require two definitions. 

\begin{definition}
\label{Def:FaceAdjacent}
Suppose that $P$ and $Q$ are distinct tetrahedron rectangles. 
We say that $P$ and $Q$ are \emph{face-adjacent} if their intersection, $P \cap Q$, is a face rectangle. 

In general, we say that two tetrahedron rectangles $P$ and $Q$ are \emph{face-connected} if there is a finite sequence $(P = P_0, P_1, \ldots, P_n = Q)$ of tetrahedron rectangles where $P_i$ and $P_{i+1}$ are face-adjacent for all $i$. 
\end{definition}

Note that every tetrahedron rectangle is face-connected to itself.

\begin{definition}
\label{Def:Spans}
Suppose that $P$ and $Q$ are rectangles of $\calL$.  
We say that $P$ \emph{west-east spans} $Q$ if there is a leaf of the induced foliation $F_Q$
that is contained in $P$.
We say that $P$ \emph{properly} west-east spans $Q$ if, additionally, $P - Q$ has two components. 
We define \emph{south-north spans} and \emph{properly south-north spans} similarly.
Finally, we say that $P$ \emph{(properly) spans} $Q$ if either $P$ (properly) west-east spans $Q$, or $P$ (properly) south-north spans $Q$.
\end{definition}

\begin{remark}
\label{Rem:Spans}
Note that the definition of west-east spans is independent of any choices of orientation made as in \refrem{Cardinal}.
\end{remark}

\vspace{-35pt}
\mbox{}
\begin{wrapfigure}[11]{r}{0.42\textwidth}
\vspace{8pt}
\centering
\labellist
\small\hair 2pt
\pinlabel  $R$ at 215 190
\endlabellist
\includegraphics[width=0.37\textwidth]{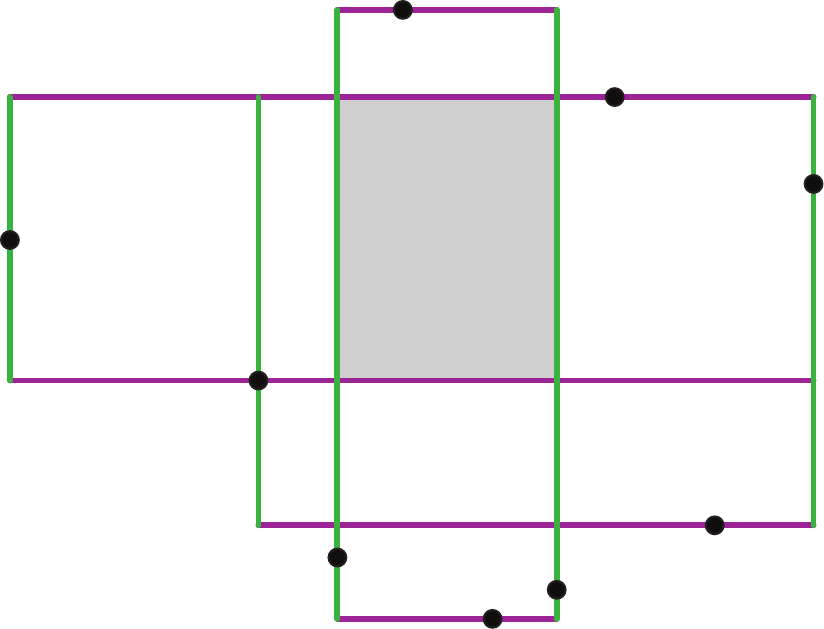}
\caption{A possible picture for \reflem{Ascend}.}
\label{Fig:Ascend}
\end{wrapfigure}

\begin{lemma}
\label{Lem:Ascend}
Suppose that $P$ and $Q$ are tetrahedron rectangles. 
Suppose that $P$ west-east spans $Q$.  
Then there is a sequence of tetrahedron rectangles $(P = P_0, P_1, \ldots, P_n = Q)$ 
so that 
\begin{itemize}
\item
$P_i$ is face-adjacent to $P_{i+1}$ and
\item
$P_i$ west-east spans $P_{i+1}$.
\end{itemize}
\end{lemma}

\begin{proof}
Let $R = P \cap Q$.
We define $\tetra(R)$ to be the set of tetrahedron rectangles that contain $R$.
By \reflem{Finite} the set $\tetra(R)$ is finite. 
We induct on the size of $\tetra(R)$.
In the base case $R = P = Q$, so $\tetra(R)$ has exactly one element and there is nothing to prove.

In general, let $F$ be any face rectangle of $P$ that south-north spans $P$ and contains $R$.
Applying \reflem{FaceTet} there is exactly one tetrahedron rectangle $P'$ that is, via $F$, face-adjacent to $P$.
Note that $P$ west-east spans $P'$ which in turn west-east spans $Q$.
See \reffig{Ascend}.
Here the widest rectangle is $P$, the tallest is $Q$, and the remaining tetrahedron rectangle is $P'$.
Set $R' = P' \cap Q$ and note that $R \subset R'$.
Thus $\tetra(R') \subset \tetra(R)$.
Furthermore, $P$ is an element of $\tetra(R)$ but is not an element of $\tetra(R')$.
The induction hypothesis now implies that $P'$ is face-connected to $Q$, using only the tetrahedra in $\tetra(R')$, completing the proof. 
\end{proof}

\begin{proposition}
\label{Prop:FaceConnected}
The set of tetrahedron rectangles of $\calL$ is face-connected.
\end{proposition}

\begin{proof}
Suppose that $P$ and $Q$ are tetrahedron rectangles.  
Choose an arc $\gamma \subset \calL$ connecting a point of $P$ to a point of $Q$.  
Recall that by \reflem{Basis} the open rectangles give a basis for the topology.
Also, $\gamma$ is compact. 
Thus $\gamma$ admits a finite covering by rectangles.
By \refdef{Loom}\refitm{Tet} the arc $\gamma$ is covered by a finite collection of tetrahedron rectangles. 

Thus we are reduced to the case where $P$ and $Q$ intersect. 
Let $R = P \cap Q$.  
Let $R'$ be the rectangle so that
\begin{itemize}
\item
$R'$ contains $R$, 
\item
the west and east sides of $R'$ contain, respectively, the west and east sides of $R$, and
\item
$R'$ is maximal with respect to the above two properties. 
\end{itemize}
From \refdef{Loom}\refitm{Tet} we deduce that $R'$ is a tetrahedron rectangle. 
From the construction we deduce that both $P$ and $Q$ west-east span $R'$ (and perhaps one or both equal $R'$). 
The proposition now follows from two applications of \reflem{Ascend}.
\end{proof}

We deduce the following. 

\begin{corollary}
\label{Cor:Countable}
There are countably infinitely many tetrahedron rectangles.  
Thus the same holds for face rectangles, edge rectangles, cusps, and cusp leaves. \qed
\end{corollary}



This implies that there are only countably many cusp leaves.  
We deduce the following. 

\begin{corollary}
\label{Cor:NonCuspLeavesDense}
The non-cusp leaves are dense in $F^\calL$ and $F_\calL$. \qed
\end{corollary}

\section{Locally veering triangulations}
\label{Sec:Veering}

In this section we review several combinatorial structures on triangulations of three-manifolds.
We introduce the notions of \emph{taut isomorphisms} and \emph{locally veering triangulations}.
We then follow Gu\'eritaud~\cite[Section~2]{Gueritaud16} to construct a locally veering triangulation from a loom space.
In \refprop{Functorial} we show that this construction is functorial.

\subsection{Definitions}
\label{Sec:Definitions}

A useful example for the first several definitions is the canonical triangulation of the figure-eight knot complement.  
See \reffig{FigEightTri}.

\begin{wrapfigure}[19]{r}{0.3\textwidth}
\vspace{-16pt}
\captionsetup[subfloat]{captionskip=15pt}
\centering
\subfloat[Cusped model tetrahedron.]{
\labellist
\small\hair 2pt
 \pinlabel {$v_0$} [tr] at 2 4
 \pinlabel {$v_1$} [tl] at 292 96
 \pinlabel {$v_2$} [bl] at 343 330
 \pinlabel {$v_3$} [br] at 35 352
\endlabellist
\centering
\includegraphics[width = 0.2\textwidth]{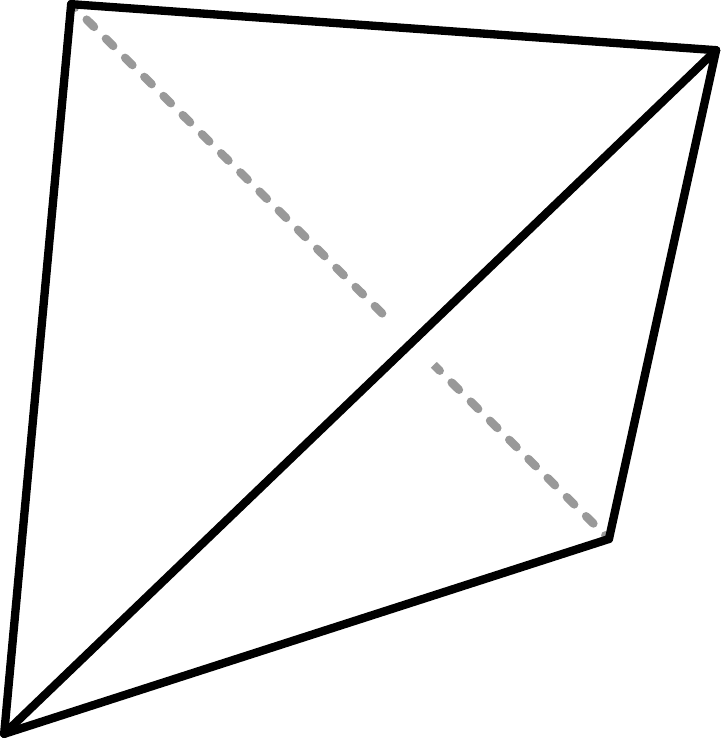}
\label{Fig:Model}
}

\subfloat[Face pairing.]{
\centering
\includegraphics[width = 0.25\textwidth]{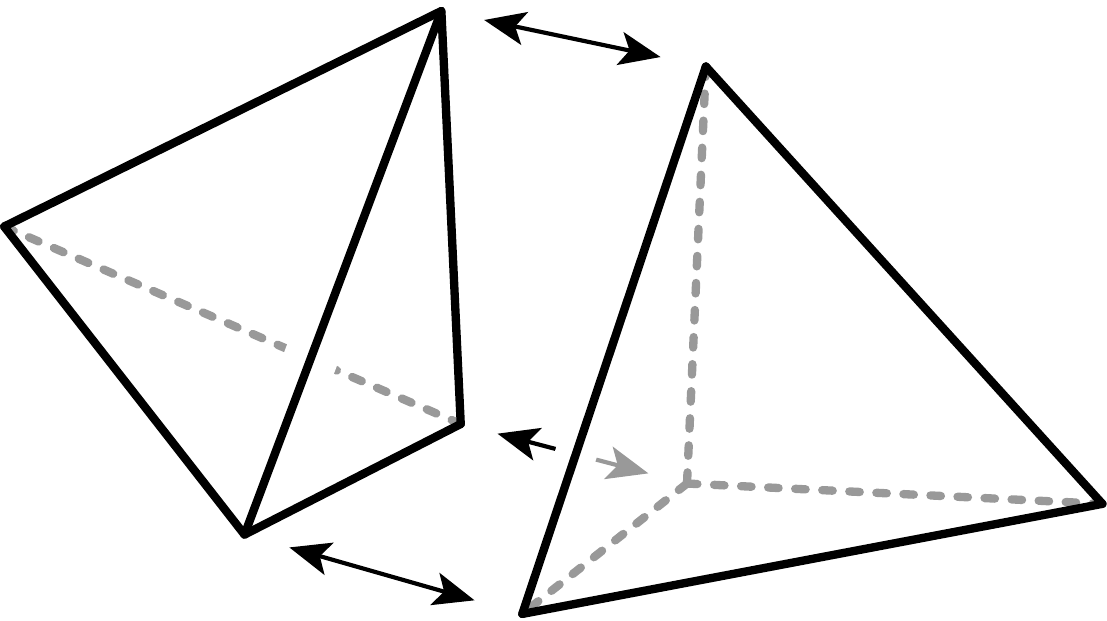}
\label{Fig:FacePairing}
}
\caption{}
\label{Fig:Models}
\end{wrapfigure}

Let
\[
t^3 = \left\{ x \in \RR^4 \st \mbox{$x_i \geq 0$ and $\sum x_i = 1$} \right\}
\]
be the \emph{standard} tetrahedron.
This is equipped with the subspace topology. 
Note that the vertices of $t^3$ are the standard unit vectors. 
Their usual ordering gives an orientation to $t^3$.

A copy $t$ of $t^3$ is called a \emph{model} tetrahedron. 
See \reffig{Model}. 
The facets (faces, edges, and vertices) of $t$ are called \emph{model facets}.
Note that $t$ also inherits an orientation. 

Suppose that $t$ and $t'$ are model tetrahedra (which may be equal).
Suppose that $f$ and $f'$ are faces of $t$ and $t'$, respectively.
Suppose that $\phi \from f \to f'$ is a homeomorphism induced by restricting an affine map.
We call $\phi$ a \emph{face pairing}.
See \reffig{FacePairing}.

Essentially following~\cite[Section~4.2]{Thurston78}, we define an \emph{ideal triangulation} $\calT = (\{t_\alpha\}, \{\phi_\beta\})$ to be a collection of model tetrahedra and a collection of face pairings. 
The \emph{realisation} of $\calT$, denoted $|\calT|$, is the topological space obtained as follows. 
\begin{itemize}
\item
Take the disjoint union of the model tetrahedra.
\item
Quotient by 
the face pairings. 
\item
Remove the zero-skeleton of the result.
\end{itemize}
The \emph{realisation} of a model facet of $\calT$ is its image in $|\calT|$.
The \emph{models} of a realised facet in $|\calT|$ are its preimages in $\calT$. 
In order to ensure that $|\calT|$ is a three-manifold we require the following. 
\begin{itemize}
\item
If $\phi$ is a face pairing then so is $\phi^{-1}$.
\item
Every model face occurs in exactly two face pairings.
\item
No face is paired with itself. 
\item
Every edge has only finitely many models. 
\item
The models of a single edge can be consistently oriented. 
\end{itemize}

A \emph{taut} structure on a model tetrahedron $t$ is an assignment of dihedral angle of either zero or $\pi$ to the model edges of $t$. 
The dihedral angles are required to satisfy the following. 
\begin{itemize}
\item
Suppose that $v$ is a model vertex of $t$. 
Suppose that $e, e', e''$ are the model edges of $t$ adjacent to $v$. 
Then the sum of their dihedral angles is $\pi$.
\end{itemize}
In a taut tetrahedron, the edges with dihedral angle zero are called \emph{equatorial} while the edges with dihedral angle $\pi$ are called \emph{diagonal}.  
See \reffig{VeeringTet}.

Following~\cite[Definition~1.1]{HRST11} (see also~\cite[page~370]{Lackenby00}) we say an ideal triangulation $\calT$ is \emph{taut} if all model tetrahedra are taut and we have the following. 
\begin{itemize}
\item
Suppose that $e$ is an edge of $|\calT|$.
Then the dihedral angles of the models of $e$ sum to $2\pi$. 
\end{itemize}

\begin{definition}
\label{Def:TautIsom}
Suppose that $\calT$ and $\calS$ are taut ideal triangulations. 
If $f \from \calT \to \calS$ is an isomorphism of triangulations 
and sends the taut structure on $\calT$ to that on $\calS$, 
then we call $f$ a \emph{taut isomorphism}.
\end{definition}

Again, taut isomorphisms compose in the usual way. 
Thus taut triangulations, together with taut isomorphisms, form a category denoted $\Taut$.
We use $\Isom(\calT, \calS)$ to denote the set of taut isomorphisms from $\calT$ to $\calS$; 
we use $\Aut(\calT)$ to denote the group of taut automorphisms. 
It is an exercise to check that the taut automorphism group of the triangulation of the figure-eight knot complement, shown in \reffig{FigEightTri}, is isomorphic to the symmetries of the square.


An oriented taut tetrahedron $t$ is \emph{veering} if there is a bi-colouring (by red and blue) of the model edges as follows. 
\begin{itemize}
\item
Suppose that, at a model vertex $v$, we have adjacent model edges $e, e', e''$. 
Suppose that this is the anticlockwise ordering on the edges, as viewed from outside of $t$ and using the induced orientation on $\bdy t$.  Then if $e$ has dihedral angle $\pi$ we have that $e'$ is blue and $e''$ is red. 
\end{itemize}
See \reffig{VeeringTet}.

\begin{wrapfigure}[13]{r}{0.33\textwidth}
\vspace{-30pt}
\centering
\labellist
\small\hair 2pt
 \pinlabel {$0$} [tr] at 69 38
 \pinlabel {$0$} [tl] at 165 54
 \pinlabel {$0$} [bl] at 138 154
 \pinlabel {$0$} [br] at 41 135
 \pinlabel \textcolor{mygray}{$\pi$} [t] at 53 79
 \pinlabel {$\pi$} [r] at 101 161
\endlabellist
\includegraphics[width=0.3\textwidth]{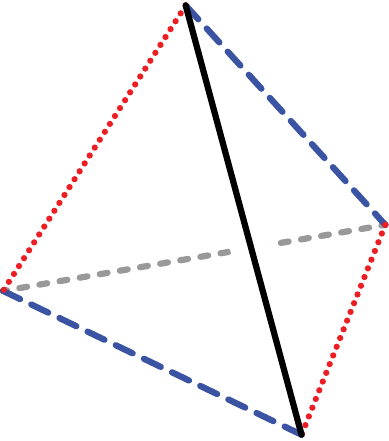}
\vspace{-5pt}
\caption{A veering tetrahedron. The $\pi$ angle edges may be either red or blue.}
\label{Fig:VeeringTet}
\end{wrapfigure}
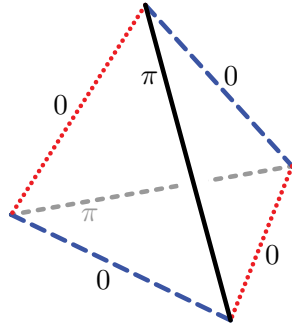

Suppose now that we have fixed an orientation of a taut ideal triangulation $\calV$. 
Following~\cite[Definition~1.3]{HRST11} (see also~\cite[Definition~4.1]{Agol11}) we define a \emph{veering} structure on $\calV$ to be a colouring (by red and blue) of the edges of $|\calV|$ that pull back to give veering structures on all of the oriented model taut tetrahedra.
For an example, see \reffig{FigEightTri}.

We now turn to generalising veering triangulations to non-orientable manifolds. 
Suppose that $M$ is a three-manifold. 
Suppose that $\calT$ is a taut ideal triangulation of $M$.
Suppose that $e$ is an edge of $\calT$. 
Let $(e_i)$ be the collection of models of $e$.  
We order the $e_i$ cyclically, as we walk about $e$ in $M$. 
For each model edge $e_i$ let $t_i$ be a copy of the model tetrahedron containing $e_i$.
For each $i$ let $\phi_i$ be the face pairing, from $f_i \subset t_i$ to $g_i \subset t_{i+1}$, 
so that $\phi_i(e_i) = e_{i+1}$.
We define $\calT_e = (\{t_i\}, \{\phi_i\})$ to be the \emph{model edge neighbourhood} of $e$ in $\calT$. 
Note that $\calT_e$ inherits a taut structure from $\calT$.
Also, its realisation $|\calT_e|$ is a three-ball.

Choose an orientation on $|\calT_e|$.  
We say that $\calT$ is \emph{veering at $e$} if $\calT_e$ admits a veering colouring. 

\begin{definition}
\label{Def:LocallyVeering}
Suppose that $M$ is a three-manifold.  
Suppose that $\calT$ is a taut ideal triangulation of $M$. 
Then $\calT$ is \emph{locally veering} if $\calT$ is veering at every edge. 
\end{definition}

\begin{example}
\label{Exa:Gieseking}
The Gieseking manifold $M_G$ can be obtained as a punctured torus bundle; 
the monodromy is the matrix $G$ given in \refexa{AnosovMap}.
Thus $M_G$ admits a taut ideal triangulation $\calV$ with a single tetrahedron.  
The orientation double cover of $M_G$ is the figure-eight knot complement; 
the taut triangulation $\calV$ lifts to give the one shown in \reffig{FigEightTri}.
Thus $\calV$ is locally veering. 
\end{example}

Locally veering triangulations are ``almost'' veering, in the following sense. 

\begin{proposition}
\label{Prop:LocallyVeering}
Suppose that $M$ is a three-manifold.
Suppose that $\calV$ is a taut ideal triangulation of $M$.
Suppose that $(M, \calV)$ is locally veering.  
Then $(M, \calV)$ admits a veering structure if and only if $M$ is orientable. 
\end{proposition}

\begin{proof}
Suppose that $(M, \calV)$ admits a veering structure.  
Then each model tetrahedron is veering, thus oriented. 
These orientations disagree across faces (because of the colours of the edges). 
Thus they define an orientation of $M$.

For the converse we fix an orientation of $M$.
This induces an orientation on all tetrahedra and thus on all model tetrahedra.
The taut structure picks four of the six model edges to be equatorial.
The orientation of the model tetrahedron and the axioms of a veering triangulation determines the colours of the equatorial model edges.
The hypothesis of local veering tells us that all equatorial models of a given edge have the same colour.
\end{proof}


We use $\Veer$ to denote the full subcategory of $\Taut$ consisting of locally veering triangulations.
We use $\Veer(\RR^3)$ to denote the further full subcategory of those triangulations whose realisation is homeomorphic to $\RR^3$.

\subsection{Building the triangulation}
\label{Sec:Construction}

We give a (metric-free) version of Gu\'eritaud's construction~\cite[Section~2]{Gueritaud16}.  
That is, for every loom space $\calL$ we give a locally veering triangulation $\calV = \veer(\calL)$.

\begin{definition}
\label{Def:InducedTriangulation}
Suppose that $\calL$ is a loom space. 
Its \emph{induced} triangulation $\veer(\calL)$ has \emph{model cells}
\[
\{ \cell(P) \st \mbox{$P$ is a skeletal rectangle of $\calL$} \}
\]
We identify the vertices of $\cell(P)$ with the cusps of $P$.
These are distinct by \reflem{OneEdgeTwoCusps}.
If $P$ and $Q$ are skeletal rectangles with $P$ properly contained in $Q$, then the set of cusps $\Delta(P)$ is a proper subset of $\Delta(Q)$. 
For every such pair we take $\phi_{P, Q}$ to be the corresponding cell identification between $\cell(P)$ and the corresponding subsimplex of $\cell(Q)$. 
The realisation of $\veer(\calL)$ is the resulting quotient (minus the zero-skeleton). 
\end{definition}

\begin{definition}
\label{Def:InducedMap}
Suppose that $f \from \calL \to \calM$ is a loom isomorphism.  
We define the \emph{induced} map $\veer_f \from \veer(\calL) \to \veer(\calM)$ as follows. 
Suppose that $P \subset \calL$ is a skeletal rectangle. 
Let $\cell(P)$ be the corresponding cell of $\veer(\calL)$; 
so $\cell(P)$ is either an edge, face, or tetrahedron. 
We take $\veer_f(\cell(P)) = \cell(f(P))$. 
\end{definition}

We also use the notation $\veer(f)$ for $\veer_f$.

\subsection{Induced triangulations}
\label{Sec:InducedTriangulation}

We present a few properties of induced triangulations.

\begin{lemma}
\label{Lem:Manifold}
Suppose that $\calL$ is a loom space. 
Let $\calV = \veer(\calL)$ be its induced triangulation.
Then the realisation $|\calV|$ is a non-compact, connected three-manifold.
Furthermore $|\calV|$ is orientable.
\end{lemma}

\begin{proof}
By \refcor{Countable} the triangulation $\veer(\calL)$ has infinitely many model tetrahedra. 
We deduce that $|\calV|$ is non-compact.
By \refprop{FaceConnected} we have that $|\calV|$ is connected.
By \reflem{FaceTet} we have that every face of $|\calV|$ meets exactly two tetrahedra of $|\calV|$. 
By \refcor{EdgeFace} we have that every edge of $|\calV|$ meets finitely many faces, and thus finitely many tetrahedra, of $|\calV|$. 
Since tetrahedron rectangles are embedded in $\calL$, 
no tetrahedron is glued to itself. 
Recall also that we removed the zero-skeleton from $|\calV|$. 
We deduce that $|\calV|$ is a non-compact, connected topological space 
which is a three-manifold away from the midpoints of edges. 

\begin{figure}[htbp]
\centering
\subfloat[The face rectangle $R$ is shaded.]{
\includegraphics[width=0.3\textwidth]{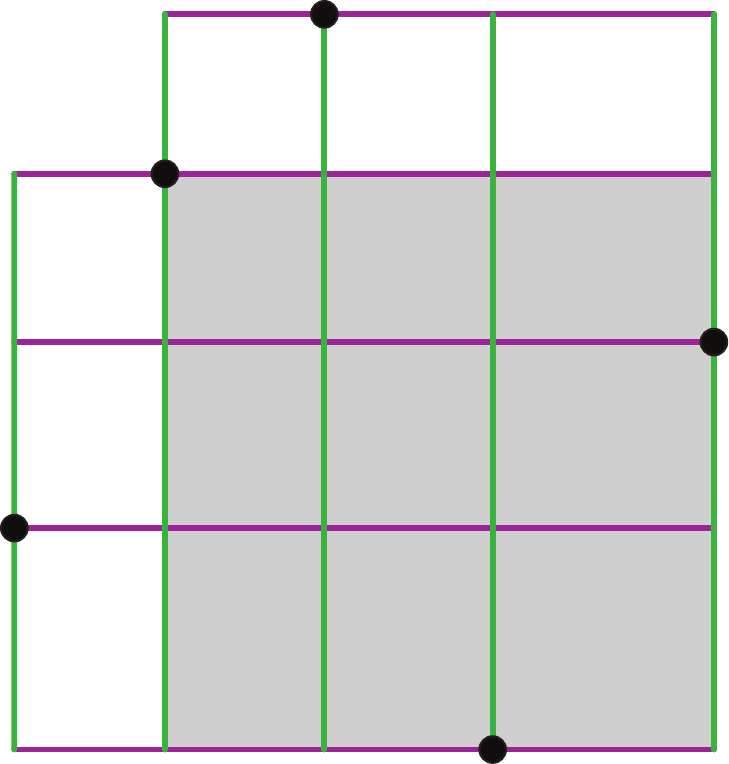}
\label{Fig:TetPairingLoom}
}
\qquad
\subfloat[The face $f_R$ is shaded.]{
\includegraphics[width=0.3\textwidth]{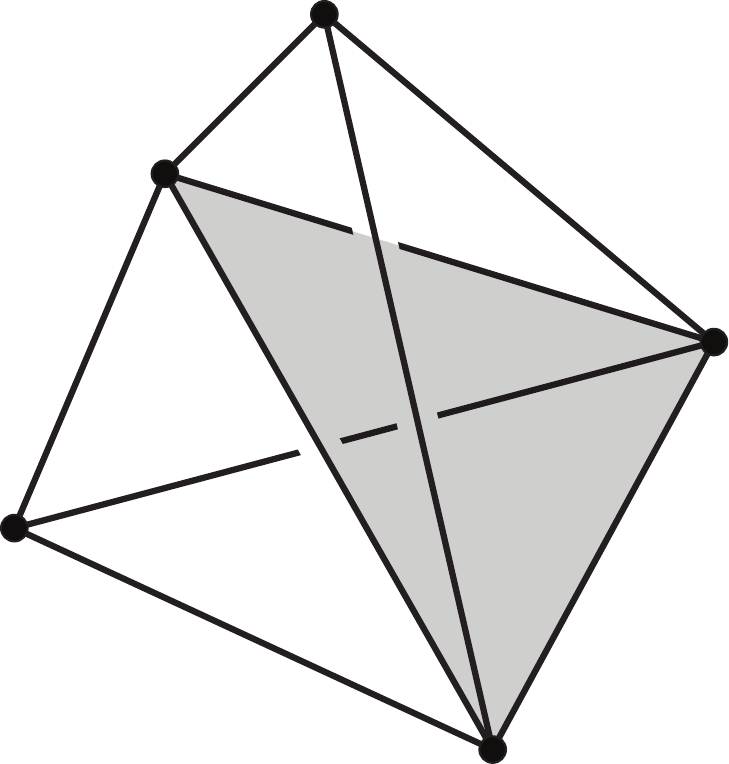}
\label{Fig:TetPairingTri}
}
\caption{The induced orientations on the shared face are opposite.}
\label{Fig:TetPairing}
\end{figure}

As in \refrem{Cardinal}, we fix orientations of $F^\calL$ and $F_\calL$.
Ordering $F^\calL$ before $F_\calL$, these orientations determine an orientation of $\calL$ as well as the cardinal directions south, east, north, and west.

Suppose that $P$ is a tetrahedron rectangle in $\calL$. 
We order the sides of $P$ according to their direction: first south, then east, north, and west. 
This induces an ordering on the model vertices of $\cell(P)$ and thus induces an orientation on $\cell(P)$.
Suppose that $Q$ is another tetrahedron rectangle which is face-adjacent to $P$. 
Suppose that $R = P \cap Q$ is the shared face rectangle.
We note that the composition $\phi_{R, P} \circ \phi_{R, Q}^{-1}$ reverses orientation.  
See \reffig{TetPairing}.
Thus the midpoints of edges are also manifold points of $|\calV|$;
also our choices of orientations above determine an orientation of $|\calV|$.
\end{proof}

Before discussing the induced taut structure on $\veer(\calL)$ we require a lemma. 


\begin{lemma}
\label{Lem:OneUpOneDown}
Suppose that $P$ is an edge rectangle of $\calL$. 
Then there are exactly two tetrahedron rectangles which contain $P$ and properly span $P$.
\end{lemma}

\begin{proof}
Let $a$ and $b$ be the cusps of $P$.
Let $m_a$, $\ell^a$, $m_b$, and $\ell^b$ be the sides of $P$, contained in the associated cusp leaves. 
Now let $R_P$ be the union of all rectangles that contain both $m_a$ and $m_b$.
Note that every rectangle in this union properly south-north spans $P$.
Appealing to \refdef{Loom}\refitm{Cusp} twice and \reflem{Basis}, the union $R_P$ is non-empty. 
See \reffig{EdgeTopBottom}.

Since $R_P$ can be realised as an increasing union of rectangles, it is a rectangle.
Furthermore, $R_P$ properly south-north spans $P$. 
Finally, since $R_P$ is maximal, it is a tetrahedron rectangle, by \refdef{Loom}\refitm{Tet}. 
The same construction, applied to $\ell^a$ and $\ell^b$, produces the tetrahedron rectangle $R^P$.

Suppose that $Q$ is a tetrahedron rectangle containing $P$, and either south-north or west-east spanning $P$.
Thus $Q$ is either contained in $R_P$ or in $R^P$, and we are done.
\end{proof}

\begin{wrapfigure}[10]{r}{0.36\textwidth}
\vspace{-5pt}
\centering
\labellist
\small\hair 2pt
 \pinlabel {$a$} [r] at 1 113
 \pinlabel {$b$} [l] at 347 202
\endlabellist
\vspace{-10pt}
\includegraphics[width=0.25\textwidth]{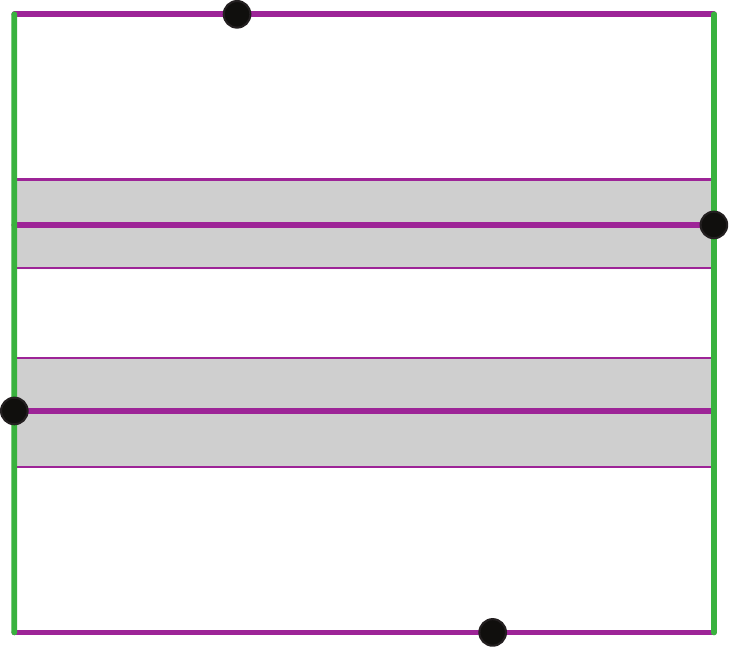}
\caption{The small rectangular neighbourhoods of $m_a$ and $m_b$ are shaded.}
\label{Fig:EdgeTopBottom}
\end{wrapfigure}

We now define the \emph{induced} dihedral angle assignment on $\calV = \veer(\calL)$. 
Suppose that $R$ is a tetrahedron rectangle in $\calL$; 
let $\cell(R)$ be the corresponding model tetrahedron. 
By \reflem{TetFaceEdge} there are six edge rectangles in $R$.
By \refdef{TetRect}, exactly two of these span $R$.
We give $\cell(R)$ a taut structure as follows.  
Suppose that $P$ is an edge rectangle contained in $R$. 
Thus $\phi_{P, R}$ gives a model edge $e$ of $\cell(R)$. 
We give $e$ dihedral angle $\pi$ or zero exactly as $P$ does or does not span $R$. 

From \reflem{OneUpOneDown} we deduce the following.

\begin{corollary}
\label{Cor:Taut}
The induced dihedral angle assignment on $\veer(\calL)$ is a taut structure. \qed
\end{corollary}

We now define the \emph{induced} colouring of the one-skeleton of $\veer(\calL)$.
Suppose that $P$ is an edge rectangle, coloured according to \refdef{EdgeRect}. 
Then we give the edge $\cell(P)$ that same colour.

Consulting \reffig{VeeringTet}, this gives us the following.

\begin{corollary}
\label{Cor:Veering}
Orienting the foliations $F^\calL$ and $F_\calL$ induces a veering structure on $\veer(\calL)$.  
Thus $\veer(\calL)$ is locally veering. \qed
\end{corollary}

\subsection{Functorial}
\label{Sec:Functorial}

We summarise this section as follows. 

\begin{proposition}
\label{Prop:Functorial}
Gu\'eritaud's construction 
\[
\veer \from \Loom(\RR^2) \to \Veer
\]
is a functor from the category of loom spaces to the category of locally veering triangulations. 
\end{proposition}

\begin{proof}
Suppose that $\calL$ is a loom space. 
By \reflem{Manifold} its induced triangulation $\veer(\calL)$ is an ideal triangulation of a non-compact, connected, orientable three-manifold. 
By \refcor{Taut} the induced dihedral angle makes $\veer(\calL)$ into a taut triangulation. 
By \refcor{Veering} we have that $\veer(\calL)$ is locally veering. 

Suppose now that $\calM$ and $\calN$ are also loom spaces.
Suppose that $f \from \calL \to \calM$ and $g \from \calM \to \calN$ are loom isomorphisms.
Recall that we use the notations $\veer_f = \veer(f)$ to represent the induced map.

Suppose that $P \subset \calL$ is a skeletal rectangle. 
If $f = \Id_\calL$ then, appealing to \refdef{InducedMap}, we have $\veer_f(\cell(P)) = \cell(f(P)) = \cell(P)$. 
Thus $\veer(\Id_\calL) = \Id_{\veer(\calL)}$, as desired.

For general loom isomorphisms $f$ and $g$, and again appealing to \refdef{InducedMap}, we have
\begin{align*}
\veer_{g \circ f}(\cell(P)) &= \cell(  g( f(P) )  ) \\
                       &= \veer_g (  \cell( f(P) )  ) \\
                       &= (\veer_g \circ \veer_f) ( \cell(P) ) 
\end{align*}
Thus $\veer(g \circ f) = \veer(g) \circ \veer(f)$.  

Note that $f^{-1} \from \calM \to \calL$ is a loom isomorphism;
thus $\veer(f^{-1})$ is an inverse for $\veer(f)$. 
We deduce that $\veer(f)$ is a bijection. 
Note that the triangulations $\veer(\calL)$ and $\veer(\calM)$ are defined solely in terms of the combinatorics of $\calL$ and $\calM$.  
Thus the cell structures, angle assignments, and local veering structures are also preserved by $\veer(f)$.
\end{proof}

We deduce that $\veer \from \Aut(\calL) \to \Aut(\veer(\calL))$ is a group homomorphism. 
In fact, as discussed in \refsec{Related}, it is an isomorphism.

\section{Convexity}
\label{Sec:Convex}

In this section we prove \refthm{ThreeSpace}: the realisation of the triangulation $\veer(\calL)$ is homeomorphic to $\RR^3$.

\begin{remark}
\label{Rem:Bundle}
One proof of \refthm{ThreeSpace} runs along the following lines. 
Choose transverse measures of full support for $F^\calL$ and $F_\calL$. 
This gives $\calL$ an incomplete, locally euclidean, metric.
Suppose that $t \in \veer(\calL)$ is a model tetrahedron. 
By sending its vertices to the associated cusps (in the completed metric), 
and then extending to all of $t$ via barycentric coordinates, we obtain a linear map from $t$ to $\calL$.
We glue these linear maps together to obtain a piecewise-linear map $\pi \from |\veer(\calL)| \to \calL$. 
We now claim the following. 
\begin{itemize}
\item The map $\pi$ is continuous and surjective.
\item Point preimages under $\pi$ are copies of $\RR$. 
\item The map $\pi$ is a fibre bundle map. 
\end{itemize}
The proof of surjectivity requires the astroid lemma (\reflem{Astroid}\refitm{DoesNot}).
Since $\calL$ is homeomorphic to $\RR^2$ we deduce that $|\veer(\calL)|$ is isomorphic to a product.
Thus the claims imply \refthm{ThreeSpace}.

In their work, Landry, Minsky, and Taylor~\cite[Proposition~5.11]{LandryMinskyTaylor22} carry this strategy out but in a more delicate setting. 
They begin with a pseudo-Anosov flow without perfect fits. 
After drilling, the leaf space is a loom space, equipped with an action by the fundamental group of the drilled three-manifold. 
Here one cannot simply choose measures since, in general, measures invariant under the action need not exist.
\end{remark}

In the remainder of \refsec{Convex} we give a very different proof.
We develop synthetic notions of geodesicity and convexity in $\calL$.
We use these to prove that $|\veer(\calL)|$ admits an exhaustion by three-balls.
The synthetic approach is longer than that of \refrem{Bundle};
however it is constructive 
and is more revealing of the structure of $\calL$. 
In future work we use this combinatorial structure to give an algorithm for drilling veering triangulations along flow loops.



\subsection{Geodesics and convexity}
\label{Sec:Geodesics}

We define a \emph{polygonal geodesic} in $\calL$. 
These are \emph{polygonal paths} (as in \cite[Definition~3.1]{Fenley12}) with additional properties.  

Our geodesics are very similar to ``staircase'' geodesics in $\RR^2$ when it is equipped with the $L^1$ metric.
However, our definition is combinatorial; 
we do not make use of a metric on $\calL$.

\begin{definition}
\label{Def:PolyPathsAndCorners}
A \emph{segment} in $\calL$ is a subarc of a leaf of $F^\calL$ or $F_\calL$.
Suppose that $a$ and $b$ lie in $\calL \cup \Delta(\calL)$. 
A \emph{polygonal path} $\gamma$ from $a$ to $b$ is 
\begin{itemize}
\item
a finite sequence $(\rho_i)_{i = 0}^{n-1}$ of oriented segments and
\item
a finite sequence $(v_i)_{i = 0}^{n}$ of material points and cusps
\end{itemize}
where
\begin{itemize}
\item
$a = v_0$, $v_n = b$, and 
\item
the arc $\rho_i$ emanates from $v_i$ and ends at $v_{i+1}$.
\end{itemize}
We say that $v_i$ is a \emph{corner} of $\gamma$ if there is a rectangle $R$ with consecutive sides contained in $\rho_{i-1}$ and $\rho_i$, respectively. 

A corner $v_i$ is \emph{right-turning} if $R$ is locally to the right of $\gamma$.  
We define \emph{left-turning} corners similarly. 
\end{definition}

Suppose that $\gamma$ is a polygonal path in $\calL$.  
Then the cusps of $\gamma$ (that is, the elements of $\Delta(\gamma)$) lie among the endpoints $v_i$ of the segments $\rho_i$.
In particular, the cusps of $\gamma$ need not appear only at the corners of $\gamma$.

\begin{definition}
\label{Def:Geodesics}
We say that a polygonal path $\gamma$ is a \emph{geodesic} if
the intersection of $\medcup_i \rho_i$ with any leaf is connected. 
\end{definition}

\begin{definition}
\label{Def:Convex}
Suppose that $H$ is a subset of $L$.
Then $H$ is \emph{convex} if, for any $p, q \in H \cup \Delta(H)$, we have that $H$ contains all geodesics between $p$ and $q$.
\end{definition}

\subsection{U-turns}

We will need a criterion for a path being a geodesic.  
We begin as follows. 

\begin{definition}
\label{Def:UTurns}
We say that a rectangle $R$ is a \emph{U-turn for $\gamma$} if
\begin{itemize}
\item 
three sides of $R$ are contained in $\gamma$ and
\item 
$R \cap \gamma$ is empty.
\end{itemize}
We say that a U-turn for $\gamma$ is \emph{maximal} if it is not properly contained in another U-turn for $\gamma$. 
See \reffig{UTurns} for several examples of maximal U-turns.
\end{definition}

\begin{figure}[htbp]
\subfloat[]{
\includegraphics[height = 1.5 cm]{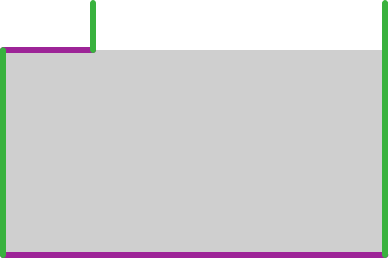} 
}
\quad
\subfloat[]{
\includegraphics[height = 1.5 cm]{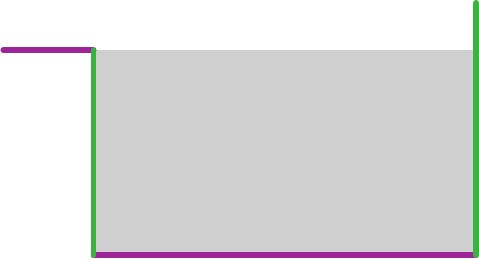} 
}
\quad
\subfloat[]{
\includegraphics[height = 1.5 cm]{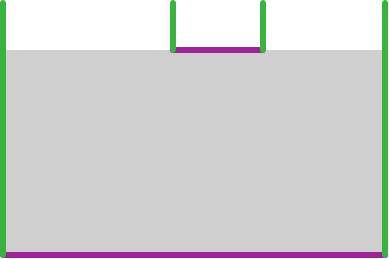} 
}
\quad
\subfloat[]{
\includegraphics[height = 1.5 cm]{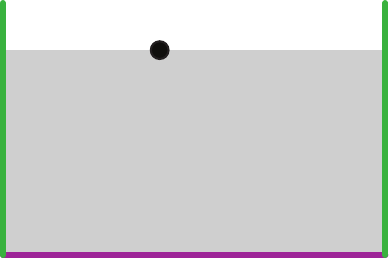} 
\label{Fig:UTurnCusp}
}

\caption{In each case the maximal U-turn is the shaded rectangle.  
Any combinatorial possibility may be obtained by combining these, 
and adding cusps to $\gamma$.}
\label{Fig:UTurns}
\end{figure}

An Euler characteristic argument gives the following.

\begin{lemma}
\label{Lem:Ear}
For any embedded polygonal loop $\gamma$ in $\calL$ and for any segment $\sigma$ of $\gamma$, there is a U-turn of $\gamma$ disjoint from $\sigma$. \qed
\end{lemma}


We use this to prove the following. 

\begin{lemma}
\label{Lem:NoUTurns}
An embedded polygonal path $\delta$ is a geodesic if and only it contains no U-turns.
\end{lemma}

\begin{proof}
If $R$ is a U-turn for $\delta$ then some leaf $\ell$ crossing the interior of $R$ intersects $\delta$ twice.

For the other direction, suppose that $\ell$ is a leaf that intersects $\delta$ twice.
Thus there is a segment $\sigma \subset \ell$ so that 
\begin{itemize}
\item
$\sigma \cap \delta = \bdy \sigma$ and 
\item
$\sigma \cup \delta$ contains an embedded polygonal loop $\gamma$.
\end{itemize}
We apply \reflem{Ear} to $\gamma$ and find a U-turn in $\delta$.
\end{proof}

\begin{lemma}
\label{Lem:Geodesic}
Suppose that $a$ and $b$ lie in $\calL \cup \Delta(\calL)$. 
Then there is a geodesic $\delta$ from $a$ to $b$. 
\end{lemma}

\begin{proof}
Using \reflem{Basis} or \refdef{Loom}\refitm{Cusp}, we choose rectangles $A$ and $B$ that have $a$ and $b$, respectively, in the interior of one of their sides. 
Since $\calL$ is path connected, there is a compact arc $\epsilon$ connecting a point of $A$ to a point of $B$. 
We cover $\epsilon$ by finitely many rectangles. 
The union of these rectangles with $A$ and $B$ contains an embedded polygonal path $\gamma$ from $a$ to $b$. 

By \reflem{NoUTurns} it now suffices to ``straighten'' $\gamma$ so that it contains no U-turns.
To this end, we define the \emph{complexity} of $\gamma$ to be the pair
\[ 
\mbox{(number of corners, number of maximal U-turns)}
\] 
ordered lexicographically.

We now induct on the complexity. 
Suppose that $\gamma$ is a polygonal path from $a$ and $b$.
Suppose that $R$ is a maximal U-turn for $\gamma$. 
Define
\[
\gamma' = (\gamma - \bdy{R}) \cup (\bdy R - \gamma)
\]
(Again, see \reffig{UTurns}.)
If $\gamma'$ is empty then $a = b$ and the desired geodesic path has no arcs. 
If not, then $\gamma'$ is a disjoint union of a polygonal path from $a$ to $b$ and some number of polygonal loops.
Let $\gamma''$ be the polygonal path and let $\{\gamma_i\}$ be the polygonal loops. 

We now claim that $\gamma''$ has lower complexity than $\gamma$.  
If $\gamma''$ has fewer corners than $\gamma$ we are done. 
If $\gamma''$ has the same number of corners as $\gamma$ then $\{\gamma_i\}$ is empty. 
Let $\sigma$ be the side of the maximal U-turn $R$ not contained in $\gamma$.  
We deduce that the interior of $\sigma$ is disjoint from $\gamma$. 
We further deduce that $\sigma$ contains a cusp in its interior; 
see \reffig{UTurnCusp}.
Thus $\gamma''$ has fewer maximal U-turns than $\gamma$. 
\end{proof}

\subsection{Sectors}

Properly speaking, cusps of $\calL$ are not points of $\calL$. 
Nonetheless we would like to treat points and cusps on an equal footing.
We do so as follows.

\begin{definition}
\label{Def:Meets}
Suppose that $x$ is a point or a cusp of $\calL$.
Suppose that $Q$ is a subset of $\calL$. 
We say that $Q$ \emph{meets} $x$ if either $x$ is a point lying in $Q$ or $x$ is a cusp of $Q$. 
\end{definition}

The following is similar to Fenley's definition of \emph{quarters}~\cite[page~22]{Fenley12}.

\begin{definition}
\label{Def:Sectors}
Suppose that $x$ is a point or cusp of $\calL$.
Let $\Lambda(x)$ be the union of all leaves meeting $x$.
We call the components of $\calL - \Lambda(x)$ the \emph{sectors based at $x$}. 
Two sectors based at $x$ are \emph{adjacent} if they are disjoint but their closures intersect along a subarc of some leaf. 
\end{definition}

Note that a (cusp) rectangle $R$, with an (ideal) corner at $x$, determines a unique sector based at $x$, denoted $\sector(x, R)$. 
Note also that the
{(closure of the)} sector $\sector(x, R)$ contains the staircase $\stair(x, R)$. 

\begin{lemma}
\label{Lem:Sectors}
If $x$ is a point of $\calL$ then there are exactly four sectors based at $x$.  
If $x$ is a cusp of $\calL$ then there are countably many sectors based at $x$; 
these are linearly ordered by the adjacency relation.
\end{lemma}

\begin{proof}
The first follows from \reflem{Basis} and \refrem{Leaves}. 
The second follows from \refdef{Cusp}, \refdef{Loom}\refitm{Cusp}, and \refrem{CuspLeaves}.
\end{proof}

\begin{definition}
\label{Def:Between}
Suppose that $p$, $q$, and $r$ lie in $\calL \cup \Delta(\calL)$.
We say that $p$ is \emph{sector-between} $q$ and $r$ if there are sectors $Q$ and $R$ based at $p$ so that 
\begin{itemize}
\item $q$ and $r$ lie in the closures of $Q$ and $R$, respectively, and
\item $Q$ and $R$ are not equal and are not adjacent.
\end{itemize}

{A point $p$ is \emph{between} $q$ and $r$ if every neighbourhood of $p$ contains a point that is sector-between $q$ and $r$.}
\end{definition}

Note that a cusp $c$ is between $q$ and $r$ if and only if it is sector-between $q$ and $r$.

\begin{lemma}
\label{Lem:GeodesicThrough}
Suppose that $p$, $q$, and $r$ lie in $\calL \cup \Delta(\calL)$.
Then $p$ is between $q$ and $r$ if and only if there is a geodesic from $q$ to $r$ passing through $p$.
\end{lemma}

\begin{proof}
Suppose that $p$ is between $q$ and $r$. 
\reflem{Geodesic} gives us geodesics $\delta$ from $q$ to $p$ and $\epsilon$ from $p$ to $r$.  
Since $\delta$ is a geodesic that ends at $p$, it must be contained in the closure of a single sector at $p$. 
The same holds for $\epsilon$. 
For a contradiction, suppose that $\gamma = \delta \cup \epsilon$ is not a geodesic.
By \reflem{NoUTurns} we have that $\gamma$ contains a U-turn while neither $\delta$ nor $\epsilon$ do.
Thus the U-turn necessarily contains $p$.
Thus there is a leaf $\ell$ that separates $p$ from both $q$ and $r$.
Thus there is a neighbourhood of $p$ separated from $q$ and $r$ by $\ell$.
Let $p'$ be a point of $\calL$ in this neighbourhood which is sector-between $q$ and $r$.
Then $q$ and $r$ are in either the same or adjacent sectors at $p'$, a contradiction.


Before proving the opposite direction, we introduce the notion of a \emph{dogleg} of a geodesic $\rho = (\rho_i)$. 
This is a non-terminal segment $\rho_i$ of $\rho$ so that 
\begin{itemize}
\item there is a rectangle $R$ with $\rho_i$ in the interior of $R$, 
\item the ends of $\rho_i$ are on non-adjacent sides of $R$, 
\item initial segments of $\rho_{i\pm1}$ are contained in these sides, and
\item $\rho_{i\pm1}$ are separated by the leaf containing $\rho_i$.
\end{itemize}
Note that points on the dogleg $\rho_i$ are sector-between the endpoints of $\rho$.

Now, suppose that $p$ lies on a geodesic $\gamma = (\gamma_i)_{i=0}^n$ from $q$ to $r$.
As noted immediately after \refdef{Between}, if $p$ is a cusp then $p$ is sector-between $q$ and $r$ and we are done.

So suppose that $p$ is a material point.
First suppose that $p$ is a corner of $\gamma$.
Let $U$ be any small rectangular neighbourhood of $p$.
Let $V$ be the component of $U - \gamma$ inside the corner.
We may form a geodesic $\gamma'$ by flipping $\gamma$ across $V$. 
So $\gamma'$ contains a dogleg near $p$.
Any point on this dogleg is sector-between $q$ and $r$.

Finally, suppose that $p$ lies in the interior of $\gamma_i$.
If $i = 0$ then $q$ lies in the closures of two sectors at $p$. 
With the appropriate choice of sector, we see that $p$ is sector-between $q$ and $r$.
The case when $i = n$ is similar.
So suppose that $0 < i < n$. 
Since no leaf has two cusps, there is a material corner at one end of $\gamma_i$.
Breaking symmetry, suppose that the material corner is between $\gamma_i$ and $\gamma_{i+1}$.
Let $V$ be a rectangle with one side equal to $\gamma_i$ and one side strictly contained in $\gamma_{i+1}$.
We again form $\gamma'$ by flipping $\gamma$ across $V$.
Again $\gamma'$ contains a dogleg near $p$ and we are done.
\end{proof}

\subsection{Hulls}

Here we again follow, at least in spirit, Section~3 of Fenley's paper~\cite{Fenley12}; 
see in particular his notion of \emph{convex polygonal paths}~\cite[Definition~3.2]{Fenley12}.

\begin{definition}
\label{Def:Hull}
Suppose that $C$ is a finite subset of $\calL \cup \Delta(\calL)$. 
We define the \emph{hull} of $C$ as follows:
\[
\hull(C) = \{ p \in \calL \st \mbox{$p$ is between some pair of elements of $C$} \} \qedhere
\]
\end{definition}

From \reflem{GeodesicThrough} we deduce the following. 

\begin{corollary}
\label{Cor:Geodesic}
Suppose that $C$ is a finite subset of $\calL \cup \Delta(\calL)$. 
Then $\hull(C)$ is the union of all geodesics connecting elements of $C$. \qed
\end{corollary}

Set $\hull(q, r) = \hull(\{q, r\})$.
Note that $\hull(q, r)$ is closed in $\calL$.
The definition of hull implies that $\hull(C) = \medcup_{q, r \in C} \hull(q, r)$.
Thus, since $\hull(C)$ is a finite union of closed sets, it is also closed.
Recall from \refdef{CuspsOf} that $\Delta(\hull(C))$ is the set of cusps of the hull $\hull(C)$. 

\begin{lemma}
\label{Lem:HullsConvex}
Suppose that $C$ is a finite subset of $\calL \cup \Delta(\calL)$.
Suppose that $q$ and $r$ lie in $\hull(C) \cup \Delta(\hull(C))$. 
Then $\hull(q, r) \subset \hull(C)$. 
\end{lemma}

It follows that hulls are convex.
Equally well, the boundary of a hull is the union of finitely many \emph{convex polygonal paths} in the sense of~\cite[Definition~3.2]{Fenley12}.
We prove a version of this in \reflem{Interval}\refitm{Rightmost}.

\begin{proof}[Proof of \reflem{HullsConvex}]
Set $H = \hull(C)$. 
We prove the contrapositive. 
That is, we assume the following: 
\begin{itemize}
\item
$C$ is a finite subset of $\calL \cup \Delta(\calL)$, 
\item
$q$ and $r$ lie in $\calL \cup \Delta(\calL)$, and
\item
some point $p \in \hull(q, r) \subset \calL$ does not lie in $H$.
\item
Furthermore, $p$ is sector-between $q$ and $r$.
\end{itemize}
{We may make the final assumption because $H$ is closed.}

We must show that one of $q$ or $r$ lies outside of $H \cup \Delta(H)$.

As in \refrem{Cardinal} we orient $F^\calL$ and $F_\calL$ so that we may refer to the cardinal directions in $\calL$.
Let $\ell^p$ and $m_p$ be the leaves through $p$.
There are four sectors based at $p$;
breaking symmetry, we assume that $q$ lies in the closure of the south-west sector while $r$ lies in the closure of the north-east sector.

We have assumed that $p$ is not 
{between, thus is not sector-between,} any pair of points of $C$.
Breaking symmetry, we may assume that $C$ lies strictly to the east of $\ell^p$.
Since $C$ is finite, and appealing to either Corollary~\ref{Cor:CuspLeavesDense} or~\ref{Cor:NonCuspLeavesDense}, 
there is a leaf $\ell$ of $F^\calL$ that separates $\ell^p$ from $C$.
Appealing to
{\refdef{Between}}, we deduce that $H$ is east of $\ell$.

Suppose that $q$ is a point of $\calL$. 
Then $q$ is on, or west of, $\ell^p$.  
Thus $q$ is strictly west of $\ell$. 
Thus $q$ is not
{between} any pair of points of $C$, and we are done.

Suppose instead that $q$ lies in $\Delta(\calL)$.
Then no cusp rectangle at $q$ lies in $H$. 
Thus $q$ is not an element of $\Delta(H)$, and we are done.
\end{proof}

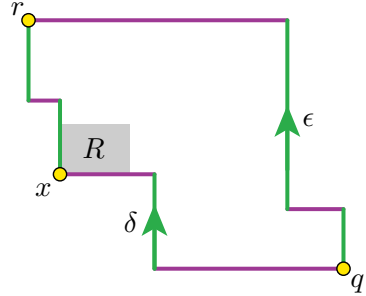
\begin{wrapfigure}[12]{r}{0.42\textwidth}
\vspace{-10pt}
\centering
\labellist
\small\hair 2pt
 \pinlabel {$q$} [tl] at 229 4
 \pinlabel {$r$} [br] at 5 181
 \pinlabel {$x$} [tr] at 24 68
 \pinlabel {$\delta$} at 77 38
 \pinlabel {$\epsilon$} at 202 110
 \pinlabel {$R$} at 51 90
\endlabellist
\includegraphics[height = 3.5 cm]{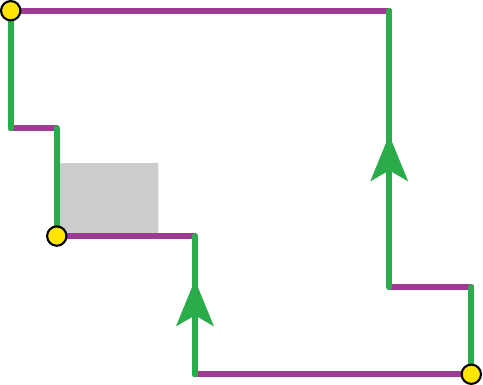}
\caption{Two geodesics $\delta$ and $\epsilon$ from $q$ to $r$.  
There is a right-turning corner (on $\delta$) at $x$.}
\label{Fig:RightOf}
\end{wrapfigure}

Suppose that $q$ and $r$ lie in $\calL \cup \Delta(\calL)$. 
Suppose that $\sector(q)$ and $\sector(r)$ are sectors, based at $q$ and $r$ respectively, so that $r$ lies in $\sector(q)$ and $q$ lies in $\sector(r)$. 
(This is possible if and only if $q$ and $r$ do not lie on a common leaf.)
Set $\sector(q, r) = \sector(q) \cap \sector(r)$.  
Note that the closure of $\sector(q, r)$ contains all geodesics from $q$ to $r$. 
Thus by \refcor{Geodesic}, the closure contains $\hull(q, r)$.

We fix, for the remainder of \refsec{Convex}, an orientation on $\calL$. 
Suppose that $\delta$ is a geodesic from $q$ to $r$.
We say that $x$, a point or cusp of $\hull(q, r)$, is \emph{to the right of} $\delta$ if $x$ meets the union of the components of $\sector(q, r) - \delta$ to the right of $\delta$.
We say that a geodesic $\epsilon$ is \emph{to the right of} $\delta$ if all points of $\epsilon - \delta$ are to the right of $\delta$.
We define \emph{to the left of} similarly.
See \reffig{RightOf}.

\begin{lemma}
\label{Lem:Interval}
Suppose that $q$ and $r$ lie in $\calL \cup \Delta(\calL)$. 
Then we have the following.
\begin{enumerate}
\item
\label{Itm:Finite}
$\Delta(\hull(q, r))$ is finite. 
\item
\label{Itm:OnGeodesic}
For every cusp $c \in \Delta(\hull(q, r))$ there is a geodesic from $q$ to $r$ running through $c$. 
\item
\label{Itm:Rightmost}
$\hull(q, r)$ contains a unique rightmost geodesic and also a unique leftmost geodesic.  
The rightmost geodesic has no right-turning corners; 
the leftmost geodesic has no left-turning corners. 
\item
\label{Itm:Boundary}
If $q$ and $r$ are cusps then the boundary of $\hull(q, r)$ is the disjoint union of the rightmost and leftmost geodesics. 
\item
\label{Itm:Disk}
If $q$ and $r$ are cusps then $\hull(q, r)$ is a disjoint finite union of finite-sided closed disks (minus finitely many boundary points).  
These disks may have cusps in common.
\end{enumerate}
\end{lemma}

\begin{proof}
\reflem{Geodesic} gives us a geodesic $\gamma$ from $q$ to $r$.
Suppose that $x$ is a right-turning corner of $\gamma$.
Let $R$ be the rectangle provided by \refdef{PolyPathsAndCorners}. 
Let $\ell^x$ and $m_x$ be the axis rays of the staircase $\stair(x, R)$.
Note that both $\ell^x$ and $m_x$ have initial segments contained in $\gamma$.

Define $m_x^\gamma$ to be the \emph{projection} of $\gamma$, to $m_x$, along $F^\calL$.
That is, suppose for $y \in \calL$ we have that $\ell^y$ is the leaf of $F^\calL$ through $y$.
Then we define $m_x^\gamma \subset m_x$ to be the polygonal path obtained by taking the closure of 
\[
 \{ m_x \cap \ell^y \st y \in \gamma \}
\]
We define $\ell^x_\gamma \subset \ell^x$ similarly.
See \reffig{Projection}. 

\begin{figure}[htbp]
\labellist
\small\hair 2pt
 \pinlabel {$x$} [tr] at 61 98
 \pinlabel {$y$} [tr] at 161 46
 \pinlabel {$c$} [bl] at 190 187
 \pinlabel {$m_x$} [l] at 284 101
 \pinlabel {$\ell^x$} [r] at 63 271
 \pinlabel {$m_c$} [r] at 1 184
 \pinlabel {$\ell^c$} [l] at 187 2
 \pinlabel {$\ell^y$} [r] at 164 2
 \pinlabel {$m_x^\gamma$} [b] at 213 117
 \pinlabel {$\ell^x_\gamma$} [l] at 80 147
 \pinlabel {$\gamma$} [t] at 225 40
\endlabellist
\includegraphics[width = .5 \textwidth]{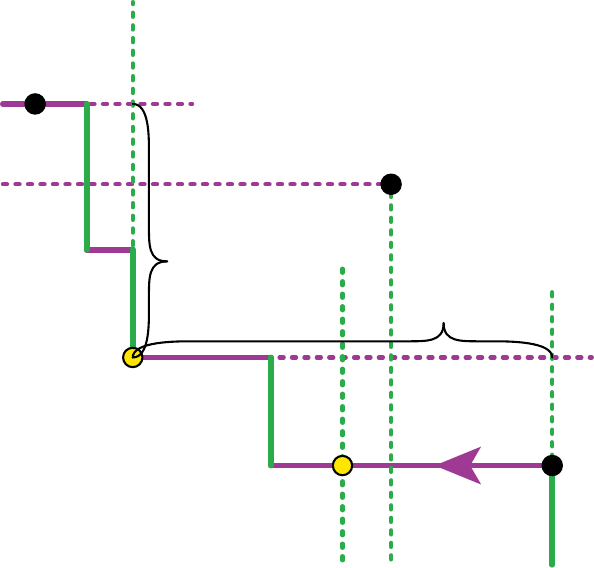}
\caption{The two projections of $\gamma$ to the axis rays of a right-turning corner $x$.}
\label{Fig:Projection}
\end{figure}

Note that since $\gamma$ is a finite polygonal path, 
both $m_x^\gamma$ and $\ell^x_\gamma$ are bounded in $m_x$ and $\ell^x$ respectively.  
(However, by \reflem{BdyQ}, the projection $m_x^\gamma$ may meet the axis cusp $c_m$ for $m_x$, if it exists.
A similar statement holds for $\ell^x_\gamma$.)
Recall that $\ExtDel(\stair(x))$ is the set of exterior cusps for $\stair(x)$. 
We define $\ExtDel^\gamma(\stair(x))$ to be those cusps of $\ExtDel(\stair(x))$ whose projections to $m_x$ and $\ell^x$ lie in $m_x^\gamma$ and $\ell^x_\gamma$ respectively. 
By the astroid lemma (\reflem{Astroid}\refitm{DoesNot}) we have that $\ExtDel^\gamma(\stair(x))$ is finite. 

Suppose that $c$ is any cusp of $\Delta(\hull(q, r))$, not on $\gamma$. 
Suppose that $c$ is to the right of $\gamma$: 
that is, $c$ lies in one of the components of $\sector(q, r) - \gamma$ to the right of $\gamma$. 
{As noted immediately after \refdef{Between},}
since $c$ is a cusp of $\Delta(\hull(q, r))$, it lies 
{sector-between} $q$ and $r$. 
Thus there are distinct, non-adjacent sectors $U$ and $V$ based at $c$ whose closures contain $q$ and $r$, respectively. 
By \reflem{Sectors} there is a sector $S$ based at $c$ separating $U$ from $V$. 
Let $m_c$ and $\ell^c$ be the cusp leaves giving the sides of $S$. 

By construction neither $q$ nor $r$ lie in $S$. 
(They may lie in the closure of $S$; that is, in $m_c$ or in $\ell^c$.)
Thus $\gamma$ intersects both $m_c$ and $\ell^c$. 
Let $S_\gamma$ be the component of $S - \gamma$ that meets $c$. 
We define $\gamma' = \gamma \cap \bdy S_\gamma$. 
Note that $\bdy S_\gamma$ is a closed polygonal loop. 

An Euler characteristic argument, applied to $\bdy S_\gamma$,
gives us a right-turning corner $x$ of $\gamma'$.
We deduce that $\ell^c$ crosses $m_x^\gamma$; 
likewise $m_c$ crosses $\ell^x_\gamma$. 
Thus $c$ lies in $\ExtDel^\gamma(\stair(x))$.
It follows that every cusp of $\Delta(\hull(q, r))$ either lies on $\gamma$ or lies in $\medcup_x \ExtDel^\gamma(\stair(x))$. 
Here $x$ ranges over the corners of $\gamma$. 
Thus $\Delta(\hull(q, r))$ is a finite union of finite sets. 
This proves \refitm{Finite}.

Suppose now that $c$ is a cusp of $\Delta(\hull(q, r))$.  
Pick any geodesic $\gamma$ from $q$ to $r$.  
If $c$ lies in $\gamma$ we are done.  
If not, then we may assume that $c$ is to the right of $\gamma$.  
Then, as above, we find $S$ and $S_\gamma$.  
Using these we define
\[
\epsilon = (\gamma - \bdy S_\gamma) \cup (\bdy S_\gamma - \gamma) 
\]
This is a geodesic through $c$, and so gives \refitm{OnGeodesic}. 

Starting with any geodesic from $q$ to $r$, \refitm{OnGeodesic} finds a geodesic to its right.
Repeating this, we find a sequence of geodesics;
by \refitm{Finite} this sequence ends with a geodesic $\zeta$ which passes through all of the rightmost cusps. 
If $\zeta$ has a right-turning corner, say at $x$, then the cusps immediately before and after $x$ span an edge rectangle.
Flipping over such edge rectangles, we obtain a new geodesic $\rho$.
Since $\rho$ has no right-turning corners, $\rho$ is rightmost.
This proves \refitm{Rightmost}.

Suppose that $\rho$ and $\lambda$ are, respectively, the rightmost and leftmost geodesics in $\hull(q, r)$.  
Suppose for a contradiction that $\rho$ and $\lambda$ intersect at a point $p$ of $\calL$.  
If $\rho$ and $\lambda$ cross at $p$ then either $\rho$ is not rightmost or $\lambda$ is not leftmost. 
Either is a contradiction. 
We deduce that near $p$ both $\rho$ and $\lambda$ lie in a single leaf, say $m_p$.  
Thus there is a cusp at the beginning of the segment $\rho \cap m_p$.  
Likewise there is a cusp at the end of the segment $\lambda \cap m_p$.  
These contradict \reflem{AtMostOne}.  This proves \refitm{Boundary}.

Appealing to the Jordan curve theorem and \refitm{Boundary} we obtain \refitm{Disk}. 
\end{proof}

\begin{corollary}
\label{Cor:StaircaseConvex}
Staircases are convex.
\end{corollary}

\begin{proof}
Suppose that $\stair(x)$ is the given staircase.
By \reflem{Astroid}, the set of exterior cusps is countable.
Let $H_k$ be the hull of $x$ together with the first $k$ exterior cusps.
By \refcor{EdgesOnStair} and \reflem{Interval}\refitm{Rightmost} the boundary of $H_k$ is contained in the closure of $\stair(x)$.
Thus the staircase is a growing union of convex sets.
\end{proof}

Lemmas \ref{Lem:Interval}\refitm{Disk} and \ref{Lem:HullsConvex} prove the following. 

\begin{corollary}
\label{Cor:Cactus}
Suppose that $C$ is a finite subset of $\Delta(\calL)$. 
Then $\hull(C)$ is a disjoint finite union of finite-sided disks 
(which may have cusps in common). \qed
\end{corollary}

\subsection{Skeletal rectangles, redux}

With convexity in hand, we are equipped to prove various existence and uniqueness results. 

\begin{lemma}
\label{Lem:ThreeEdges}
Suppose that three distinct edge rectangles meet three distinct cusps. 
Then they are all contained in a single face rectangle.
\end{lemma}

\begin{proof}
Let $a$, $b$, and $c$ be the given cusps. 
Let $A$, $B$, and $C$ be the given edge rectangles; 
by \reflem{OneEdgeTwoCusps}, we may assume that $A$ does not meet $a$, and so on.

Let $H$ be the hull of $a$, $b$, and $c$.  
By definition, the union of the closures of $A$, $B$, and $C$ gives $H$. 
By \reflem{Hull}\refitm{FiniteDisk} the hull $H$ is a disjoint finite union of
finite sided disks, adjacent only at cusps of $\Delta(H)$.
Suppose that there are two or more disks in the disjoint union.
Suppose that $c$, say, meets two of these disks. 
Then all geodesics from $a$ to $b$ run through $c$. 
In particular, all geodesics from $a$ to $b$ in the edge rectangle $C$ meet $c$. 
Thus $C$ meets $c$, a contradiction.
Thus $H$ is a single disk.  

Breaking symmetry, we suppose that $a$, $b$, and $c$ appear in anticlockwise order around $\bdy H$.
Let $\rho_a$ be the rightmost geodesic from $b$ to $c$ given by \reflem{Interval}\refitm{Rightmost}.
Define $\rho_b$ and $\rho_c$ similarly. 
Since $a$ is to the left of $\rho_a$, we deduce that $\rho_a$ lies in $\bdy H$. 
Similarly, $\rho_b$ and $\rho_c$ lie in $\bdy H$. 

Since $b$ and $c$ are the cusps of $A$, we deduce that $\rho_a$ also lies in $\bdy A$. 
So $\rho_a$, and similarly $\rho_b$ and $\rho_c$, each consist of only two segments.  
We deduce that the only cusps meeting $\bdy H$ are $a$, $b$, and $c$.
Also the only segments in $\bdy H$ are those of $\rho_a$, $\rho_b$, and $\rho_c$.

Enumerating simple polygonal loops with at most six segments, 
which occur as the boundary of a convex set, we find two possibilities. 
One has four outward corners (one at a cusp) while the other has five (two at cusps). 
These are shown in \reffig{SixSides}. 
The latter does not contain an edge rectangle between two of its cusps. 
The former is the desired face rectangle.  
\end{proof}

\begin{figure}[htbp]
\subfloat[Four outward corners.]{
\includegraphics[width = 0.3\textwidth]{Figures/face_rect} 
}
\quad
\subfloat[Five outward corners.]{
\includegraphics[width = 0.3\textwidth]{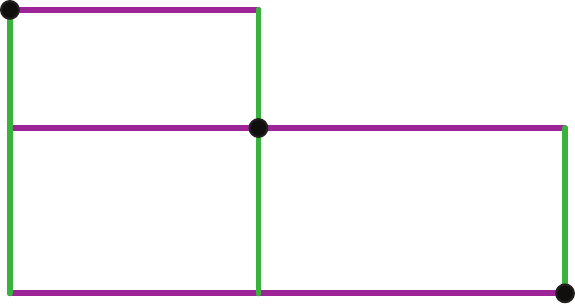} 
}

\caption{On the left we have one of the four possible face rectangles.  
Note that all four corners are outward.
On the right we have one of the four possible convex hulls with six segments in its boundary and five outward corners.}
\label{Fig:SixSides}
\end{figure}

Suppose that $P$, $Q$, and $R$ are edge, face, and tetrahedron rectangles in $\calL$, respectively. 
We say that $\cell(P)$ is the \emph{bottom edge of $\cell(Q)$} if $P$ and $Q$ west-east span each other. 
We similarly define what it means for $\cell(P)$ to be the \emph{bottom edge of $\cell(R)$}
and what it means for $\cell(Q)$ to be a \emph{bottom face of $\cell(R)$} (there are two).
We define the \emph{top} edges and faces similarly.
See \reffig{TetRect}; 
the shaded edge rectangle corresponds to the bottom edge of the corresponding faces and tetrahedron. 
Also, if $P$ is an edge rectangle then \reflem{OneUpOneDown} provides two tetrahedron rectangles: $R_P$ and $R^P$.
These are the tetrahedron rectangles where $\cell(P)$ is the bottom and top edge of $\cell(R_P)$ and $\cell(R^P)$, respectively. 

\begin{lemma}
\label{Lem:TwoEdges}
Suppose that $P$ and $Q$ are edge rectangles of $\calL$.
Suppose that $P$ properly west-east spans $Q$. 
Set
\[
H = \hull(\Delta(P) \cup \Delta(Q))
\]
Then $R_P$ and $R^Q$ lie in $H$.
\end{lemma}

\begin{proof}
Suppose that $a$ and $b$ are the cusps of $Q$. 
Breaking symmetry, suppose that $a$ is south of and $b$ is north of $P$.
Let $\ell$ be the cusp leaf emanating from $a$ and crossing $P$.
Thus either $a$ is the southern cusp of $R_P$ or $\ell$ crosses the southern side of $R_P$.
The same holds for $b$, to the north of $P$. 
Thus $R_P$ lies in $H$. 
A similar argument proves that $R^Q$ lies in $H$. 
\end{proof}

\subsection{Three-balls}

For the remainder of \refsec{Convex} we fix $T$ a non-empty, finite, and face-connected collection of tetrahedra in $\veer(\calL)$.
Note that the union of the corresponding tetrahedron rectangles need not be convex.
In a slight abuse of notation we define
\[
\Delta(T) = \medcup_{t \in T} \Delta(\cell^{-1}(t))
\]
For the remainder of \refsec{Convex} we set $H = \hull(\Delta(T))$:
that is, $H$ is the hull of the set of all cusps of the rectangles associated to tetrahedra $t \in T$. 

\begin{definition}
\label{Def:Content}
The \emph{content} of $H$, denoted $\veer(H)$, is the set of cells $k$ of $\veer(\calL)$ so that the skeletal rectangle $\cell^{-1}(k)$ lies in $H$.
\end{definition}

We next build up the combinatorial structure of $\veer(H)$ in a sequence of lemmas and corollaries.
The last of these is \refprop{ThreeBall}, proving that $\veer(H)$ is a three-ball.
From this we deduce \refthm{ThreeSpace}. 

\begin{lemma}
\label{Lem:Hull}
We have the following. 
\begin{enumerate}
\item
\label{Itm:Grows}
$T \subset \veer(H)$. 
\item
\label{Itm:FiniteContent}
$\veer(H)$ is finite. 
\item
\label{Itm:FiniteDisk}
The hull $H$ is a finite-sided closed disk in $\calL$ (minus finitely many boundary points). 
\end{enumerate}
\end{lemma}

\begin{proof}
A tetrahedron rectangle is the (interior of the) convex hull of its cusps. 
This gives~\refitm{Grows}.

By \refcor{Cactus} we know that $H$ is a finite disjoint union of finite-sided disks.
Thus $\Delta(H)$ is finite.  
Thus $\veer(H)$ is finite, giving~\refitm{FiniteContent}.

Suppose that there is more than one disk component in the union given by \refcor{Cactus}. 
We form a bipartite graph as follows: 
we take one set of nodes for the disk components, 
another set of nodes for the cusps meeting two or more disks, 
and edges for a disk-cusp pair where the latter is a cusp of the former.
By \reflem{HullsConvex} this graph is a tree.
Any leaf of the tree gives a disk that contains at least one tetrahedron rectangle $\cell^{-1}(t)$ for some $t \in T$. 
Since there are at least two disks, the tree has at least two leaves. 
We deduce that $T$ is not face-connected, a contradiction. 
This gives~\refitm{FiniteDisk}. 
\end{proof}

\begin{remark}
\label{Rem:Circular}
Recall that the loom space $\calL$ is homeomorphic to $\RR^2$.
We equip the latter with its usual (anti-clockwise) orientation, and transport this to $\calL$. 
Recall that $H = \hull(\Delta(T))$.
The orientation of $\calL$ restricts to give an orientation of $H$. 

\reflem{Hull}\refitm{FiniteDisk} says that $H$ is homeomorphic to a closed disk minus finitely many boundary points. 
The boundary of $H$ is a collection of open intervals, each meeting exactly two cusps. 
The orientation of $H$ induces a circular ordering of these intervals and thus on $\Delta(H)$.
\end{remark}

\begin{lemma}
\label{Lem:Edge}
Suppose that $a$ and $b$ are distinct cusps of $\Delta(H)$; 
suppose that $a$ and $b$ span an edge $e$ of $\veer(H)$.
Then exactly one of the following occurs:
\begin{itemize}
\item
$a$ and $b$ are adjacent in the circular order on $\Delta(H)$ or
\item
there is a tetrahedron $t \in T$ so that $\cell^{-1}(t)$ properly spans $\cell^{-1}(e)$.
\end{itemize}
\end{lemma}

\begin{proof}
Set $P = \cell^{-1}(e)$.
We orient the geodesics of $P$ from $a$ to $b$. 
Let $\lambda$ and $\rho$ be the leftmost and rightmost of these; 
note that each is connected. 

From \reflem{Hull}\refitm{FiniteDisk}, the fact that $\bdy P = \lambda \cup \rho$ has exactly two components, and the fact that $P$ is open, we deduce that $H - P$ has exactly two components.
We label these $K$ and $K'$, with $K$ to the left of $P$ and $K'$ to the right. 
Every cusp of $H$ is also a cusp of $K$ or of $K'$; 
only $a$ and $b$ are cusps of both.
Note that $K \cup P$ is convex because it is obtained by cutting $H$ (a convex set by \reflem{HullsConvex}) along $\rho$, a rightmost geodesic.
Similarly $K' \cup P$ is convex.

Suppose that $a$ and $b$ are not adjacent in the circular order on $\Delta(H)$.
Thus, each of $\Delta(K)$ and $\Delta(K')$ contains at least one cusp of $\Delta(H) - \{ a, b \}$. 
Since $K' \cup P$ is convex, the set $\Delta(K)$ contains at least one cusp from $\Delta(T) - \{a, b\}$, say $c$.
Similarly, $\Delta(K')$ contains a cusp $c'$ belonging to $\Delta(T) - \{a, b\}$.

Let $t$ and $t'$ be tetrahedra of $T$ so that $c$ and $c'$ are cusps of $R = \cell^{-1}(t)$ and $R' = \cell^{-1}(t')$, respectively.  
Since $T$ is face-connected, there is a sequence of tetrahedra 
\[
(t = t_0, t_1, t_2, \ldots, t_n = t')
\]
so that $t_i$ and $t_{i+1}$ share a face. 
Let $R_k = \cell^{-1}(t_k)$. 
Note that $R' = R_n$ is not contained in $K \cup P$, since it meets $c'$.
Let $k$ be the first index so that $R_k$ intersects $K'$.
By induction, for $i$ between $0$ and $k$, we have that $R_i$ intersects $K$. 
Thus $R_k$ intersects both $K$ and $K'$. 
Thus $R_k$ properly spans $P$. 
Thus $t_k$ is the desired tetrahedron of $T$. 
\end{proof}

\begin{wrapfigure}[6]{r}{0.39\textwidth}
\vspace{-8pt}
\centering
\labellist
\small\hair 2pt
 \pinlabel {$a$} [r] at 1 1
 \pinlabel {$m_x$} [t] at 139 1
 \pinlabel {$c$} [bl] at 100 67
 \pinlabel {$b$} [l] at 277 150
 \pinlabel {$\ell^x$} [l] at 277 75
 \pinlabel {$x$} [l] at 277 1
\endlabellist
\includegraphics[width=0.3\textwidth]{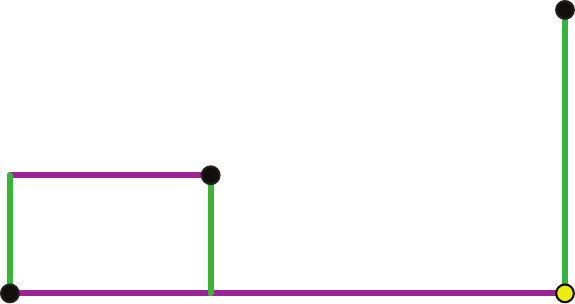}
\caption{}
\label{Fig:TooManyCoastals}
\end{wrapfigure}

We now give a partial converse to \reflem{Edge}.

\begin{lemma}
\label{Lem:Coastal}
Suppose that $a$ and $b$ are adjacent cusps in the circular order on $\Delta(H)$.
Then the interior of $\hull(a, b)$ is an edge rectangle in $\calL$.
\end{lemma}

We call the resulting edge, $e = \cell(\hull(a, b))$, a \emph{coastal edge} of $\veer(H)$.

\begin{proof}[Proof of \reflem{Coastal}]
The set $\hull(a, b)$ lies in the hull $H$ by \reflem{HullsConvex}.

Breaking symmetry, suppose that $b$ is anticlockwise of $a$ in $\bdy H$.
Applying \reflem{Interval}\refitm{Rightmost}, let $\rho$ be the rightmost geodesic from $a$ to $b$. 
Thus $\rho$ lies in $\bdy H$.
Recall also that $\rho$ has no right-turning corners.
By \reflem{AtMostOne}, the geodesic $\rho$ has at least two segments.
We deduce that $\rho$ has exactly two segments; 
these meet at one left-turning corner, say at $x \in \calL \cup \Delta(\calL)$.
However, $x$ is not a cusp by \reflem{AtMostOne}.
Let $\ell^x$ and $m_x$ be the leaves through $x$.
Breaking symmetry, we suppose that $\rho$ first runs along $m_x$ and then along $\ell^x$.
See \reffig{TooManyCoastals}.
Let $\stair(x)$ be the staircase at $x$ that meets both $a$ and $b$.

Let $\lambda$ be the leftmost geodesic from $a$ to $b$.
By \refcor{StaircaseConvex}, the staircase $\stair(x)$ contains $\lambda$.
If $\lambda$ has exactly two segments then the lemma is proved. 
For a contradiction, suppose that $\lambda$ has more than two segments.
By \reflem{Interval}\refitm{Rightmost} there is at least one cusp along $\lambda$. 
Let $c$ be the first such.
Thus $\hull(a, c)$ is the closure of an edge rectangle.

If $a$ and $c$ are adjacent in the circular order on $\Delta(H)$, then the leftmost geodesic from $a$ to $c$ is again in $\bdy H$. 
We deduce that no tetrahedron rectangle in $H$ meets $a$; see \reffig{TooManyCoastals}.
In particular, $a$ does not lie in $\Delta(T)$.
Also $a$ does not meet the hull of any pair of cusps in $\Delta(T)$.
Thus $a$ does not meet $H$, a contradiction.

Therefore $a$ and $c$ are not adjacent. 
Applying \reflem{Edge}, we find a tetrahedron $t$ so that $\cell^{-1}(t)$ properly spans the interior of $\hull(a, c)$.
Thus $\cell^{-1}(t)$ crosses one of $\ell^x$ or $m_x$. 
Either is a contradiction. 
Again, see \reffig{TooManyCoastals}.
\end{proof}

\begin{definition}
We say that $e$, an edge of $\veer(H)$, is a \emph{lower edge for $H$} if it has the following property.
For any edge $e'$ of $\veer(H)$ the edge rectangle $\cell^{-1}(e')$ does not properly west-east span $\cell^{-1}(e)$.

We say that a face $f$ of $\veer(H)$ is a \emph{lower face for $H$} if all three edges of $f$ are lower edges for $H$. 
\end{definition}

The definition implies the following. 

\begin{lemma}
\label{Lem:NoLink}
Lower edges do not link each other with respect to the given circular order on $\Delta(H)$.  \qed
\end{lemma}

From \reflem{Coastal}, we deduce the following.  

\begin{corollary}
\label{Cor:CoastalLower}
Coastal edges are lower edges for $H$.  \qed
\end{corollary}

We now show that lower edges give rise to tetrahedra in $\veer(H)$. 

\begin{corollary}
\label{Cor:BottomEdgeTet}
Suppose that $e$ is a non-coastal, lower edge for $H$.  
Then the tetrahedron $t$, having $e$ as its bottom edge, lies in $\veer(H)$.
\end{corollary}

\begin{proof}
Suppose that $e$ is the given lower edge.
Let $P = \cell^{-1}(e)$.
Define $U(e)$ to be those tetrahedra $t'$ so that $R' = \cell^{-1}(t')$ properly south-north spans $P$.  
Since $e$ is non-coastal, applying \reflem{Edge}, the set $U(e)$ is non-empty. 

Suppose that $t'$ is any tetrahedron in $U(e)$.
Again, let $R'$ be the corresponding tetrahedron rectangle. 
Since $e$ is lower, $R'$ does not properly west-east span $P$. 
Thus $R'$ instead properly south-north spans $P$.
Let $e'$ be the upper edge of $t'$. 
Let $P' = \cell^{-1}(e')$. 
We now apply \reflem{TwoEdges} to the edge rectangles $P$ and $P'$. 
The resulting tetrahedron $t = \cell(R_P)$ gives the result. 
\end{proof}

\begin{lemma}
\label{Lem:Landscape}
The lower edges and faces for $H$ form a triangulated disk in $\veer(H)$. 
\end{lemma}

\begin{proof}
Let $L(H)$ be the subcomplex of $\veer(\calL)$ consisting of the lower edges and faces for $H$. 
By \refcor{CoastalLower} all coastal edges of $V(H)$ belong to $L(H)$.
Since $T$ is non-empty, the set $\Delta(H)$ has at least four elements. 
From this and \reflem{Hull}\refitm{FiniteDisk} we deduce that $L(H)$ contains a simple edge cycle.

Suppose that $\ell$ is any simple edge cycle in $L(H)$.
By Lemmas~\ref{Lem:OneEdgeTwoCusps} and~\ref{Lem:TwoCuspsAtMostOneEdge} the edge cycle $\ell$ has length at least three.
If $\ell$ has length exactly three then, by \reflem{ThreeEdges} and by the definition of lower faces, there is a face $f$ contained in $L(H)$ spanning $\ell$.  
If $\ell$ has length greater than three then we must show that $\ell$ has a chord in $L(H)$.
This and induction then proves the lemma.

Let $H_\ell$ be the hull of the cusps of $\ell$. 
Let $e'$ be any edge of $\ell$. 
Let $P' = \cell^{-1}(e')$ be the corresponding edge rectangle. 
Thus $P'$ lies in $H_\ell$ and separates $\calL$. 
All edges of $\ell \subset L$ are lower.
Thus their edge rectangles do not properly span each other.
Therefore all cusps of $\Delta(H_\ell) - \Delta(P')$ meet a common component of $\calL - P'$.
Orient $e'$ so that this component is to the left of $P'$.
Let $\rho_{e'}$ be the rightmost geodesic in $P'$.
Since $H_\ell$ is convex, $\rho_{e'}$ lies in the boundary of $H_\ell$.

\begin{claim}
\label{Clm:HLoopDisk}
$H_\ell$ is a disk and $\bdy H_\ell = \bigsqcup_e \rho_e$ where the union ranges over the edges of $\ell$.
\end{claim}

\begin{proof}
Note that $\ell$ is simple: that is, it enters and exits each of its cusps exactly once.
Thus $\bigsqcup_e \rho_e$ meets all cusps of $\ell$; 
for each cusp exactly one arc enters the neighbourhood of that cusp and exactly one arc leaves.
Since $\bigsqcup_e \rho_e$ lies in $\bdy H_\ell$ the arcs of $\bigsqcup_e \rho_e$ do not intersect in $\calL$.
Thus $\bigsqcup_e \rho_e$ is a polygonal loop in $\calL$.
By the Jordan curve theorem, $\bigsqcup_e \rho_e$ bounds a disk $D$ in $\calL$.
By \refcor{Cactus} the hull $H_\ell$ is a finite union of finite-sided disks, meeting at cusps.  
Thus $D$ is exactly one of these disks.
The claim follows.
\end{proof}

Recall that the simple edge loop $\ell$ has length at least four. 
We now find a chord of $\ell$ in $L(H)$.

Applying \reflem{Ear} to $\bdy H_\ell$ we obtain an (open) rectangle $R \subset H_\ell$ with three consecutive sides $s$, $s'$, and $s''$ contained in $\bdy H_\ell$ and otherwise disjoint from $\bdy H_\ell$.
By \refclm{HLoopDisk} there is a cusp $b$ of $\ell$ meeting $s'$. 
Let $a$ and $c$ be the (distinct) cusps of $\bdy H_\ell$ immediately clockwise and anticlockwise of $b$, respectively.
There are now two cases as $a$ and $c$ do or do not span an edge of $\veer(\calL)$.

First suppose that there is an edge rectangle $P = \cell^{-1}(e)$ with cusps at $a$ and $c$.
See \reffig{BdyHLoopEarRectangle} for the three possible configurations (up to horizontal and vertical reflections).

\begin{figure}[htbp]
\centering
\subfloat[]{
\labellist
\small\hair 2pt
 \pinlabel {$a$} [b] at 125 193
 \pinlabel {$b$} [r] at 0 85
 \pinlabel {$c$} [l] at 273 0
 \pinlabel {$P$} at 200 120
\endlabellist
\includegraphics[width=0.27\textwidth]{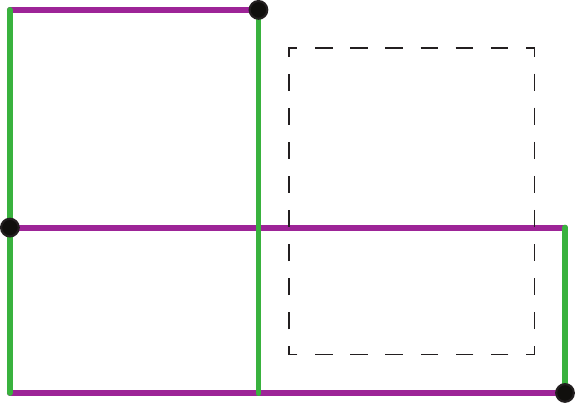}
\label{Fig:BdyHLoopEarRectangleBOnSide}
}
\quad
\subfloat[]{
\labellist
\small\hair 2pt
 \pinlabel {$a$} [b] at 125 193
 \pinlabel {$b$} [r] at 0 3
 \pinlabel {$c$} [l] at 273 80
 \pinlabel {$P$} at 200 136
\endlabellist
\includegraphics[width=0.27\textwidth]{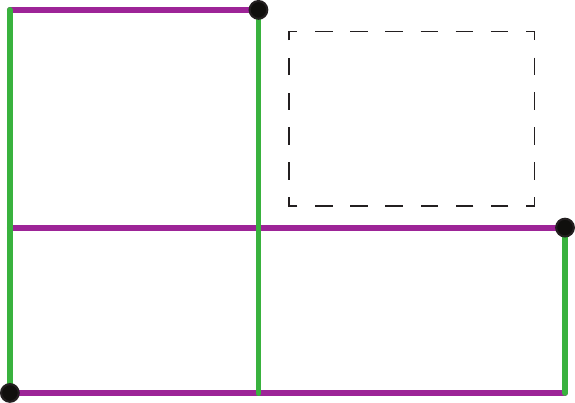}
}
\quad
\subfloat[]{
\labellist
\small\hair 2pt
 \pinlabel {$a$} [b] at 5 193
 \pinlabel {$b$} [br] at 120 8
 \pinlabel {$c$} [l] at 273 80
  \pinlabel {$P$} at 180 136
\endlabellist
\includegraphics[width=0.27\textwidth]{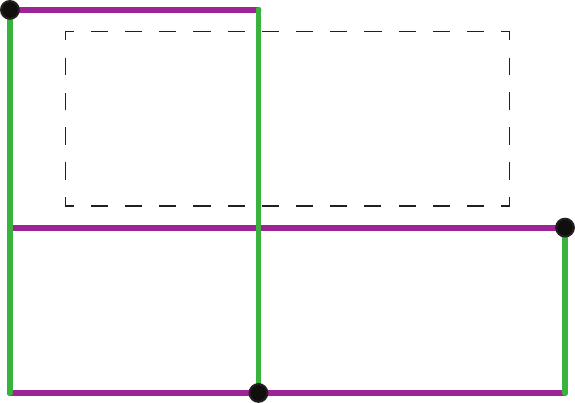}
}
\caption{}  
\label{Fig:BdyHLoopEarRectangle}
\end{figure}

Since $\ell$ has length greater than three, we deduce that $e$ is not an edge of $\ell$. 
If $e$ is lower for $H$ then we have found a chord for $\ell$ and we are done.
Suppose instead that $e$ is not lower for $H$. 
Thus there is some edge $e'$ of $\veer(H)$ so that $P' = \cell^{-1}(e')$ properly west-east spans $P$. 
Repeating the argument if needed, we may also assume that $e'$ lies in $L(H)$.
Then $P'$ cannot west-east span any edge rectangle of $\cell^{-1}(\ell)$ as these are all in $L(H)$.
The only possible configuration is shown in \reffig{BdyHLoopEarRectangleBOnSide}, where $P'$ has a cusp at $b$.
By \reflem{NoLink}, the edge $e'$ does not link any edge of $\ell$.
Thus $e'$ is a chord for $\ell$, a contradiction.

Next suppose that there does not exist an edge rectangle with cusps at $a$ and $c$.
See \reffig{BdyHLoopEarCusps} for the three possible configurations (up to horizontal and vertical reflections).

\begin{figure}[htbp]
\centering
\subfloat[]{
\labellist
\small\hair 2pt
 \pinlabel {$a$} [b] at 125 193
 \pinlabel {$b$} [r] at 0 85
 \pinlabel {$c$} [l] at 273 0
 \pinlabel {$d$} [bl] at 169 153
\endlabellist
\includegraphics[width=0.27\textwidth]{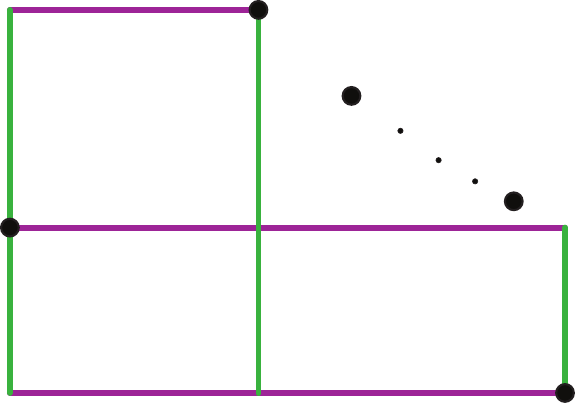}
}
\quad
\subfloat[]{
\labellist
\small\hair 2pt
 \pinlabel {$a$} [b] at 125 193
 \pinlabel {$b$} [r] at 0 3
 \pinlabel {$c$} [l] at 273 80
 \pinlabel {$d$} [bl] at 169 153
\endlabellist
\includegraphics[width=0.27\textwidth]{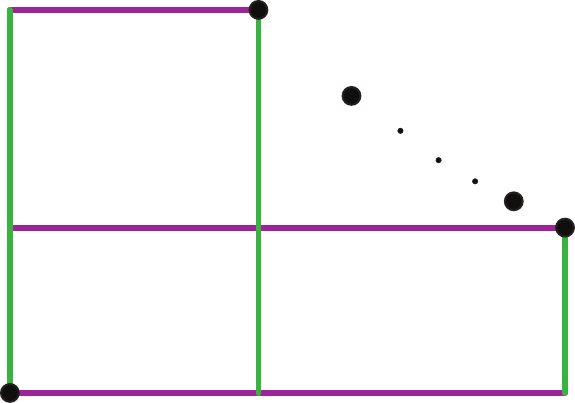}
}
\quad
\subfloat[]{
\labellist
\small\hair 2pt
 \pinlabel {$a$} [b] at 5 193
 \pinlabel {$b$} [br] at 120 8
 \pinlabel {$c$} [l] at 273 80
 \pinlabel {$d$} [bl] at 169 153
\endlabellist
\includegraphics[width=0.27\textwidth]{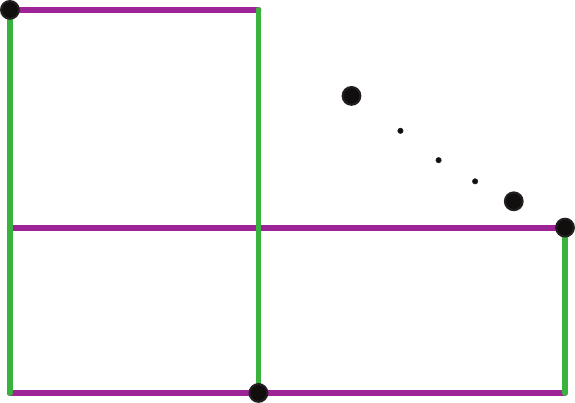}
}
\caption{}  
\label{Fig:BdyHLoopEarCusps}
\end{figure}

In all cases, the rightmost geodesic from $a$ to $c$ consists of exactly two segments. 
The leftmost geodesic from $a$ to $c$ necessarily meets at least one cusp, say $d$. 
The cusp $d$ is both an exterior cusp for the staircase based at $b$ and is a cusp of $\Delta(H)$.
Thus there is an edge rectangle $P' = \cell^{-1}(e')$ with cusps at $b$ and $d$. 
In all three cases, any edge rectangle $P''$ that strictly west-east spans $P'$ necessarily strictly west-east spans the edge rectangle $\hull(a,b)$.
However, $\cell(\hull(a,b))$ is an edge of $\ell$, thus lower for $H$.
Thus $P''$ is not an edge rectangle of $H$.
We deduce that the edge $e'$ is lower for $H$.
Thus $e'$ is a chord for $\ell$, as desired.
\end{proof}

\begin{proposition}
\label{Prop:ThreeBall}
Suppose that $T$ is a finite, face-connected, non-empty collection of tetrahedra in $\veer(\calL)$.  
Let $H = \hull(\Delta(T))$.  
Then the realisation of $\veer(H)$ is a closed three-ball (minus finitely many boundary points).
\end{proposition}

\begin{proof}
\reflem{Landscape} gives us a triangulated disk $L(H)$, whose edges and faces are lower for $H$ and whose boundary consists of the coastal edges for $H$.
Set $L_0 = L(H)$.
We now use induction to obtain a sequence of triangulated disks $L_k \subset \veer(H)$, with $\bdy L_k = \bdy L_0$.

We say that a non-boundary edge $e$ of $L_k$ is \emph{flippable} if both of its adjacent faces in $L_k$, say $f$ and $f'$, 
have $e$ as their bottom edge.  
(That is, all three of $\cell^{-1}(e)$, $\cell^{-1}(f)$, and $\cell^{-1}(f')$ west-east span each other.  See \reffig{TetRect}.)

If $L_k$ has a flippable edge then let $e_k$ be one such. 
Consulting \reffig{TetRect}, we see that $f$ and $f'$ are the bottom faces of a tetrahedron $t_k$ in $\veer(H)$.
Let $g$ and $g'$ be the top faces of $t_k$. 
We define the result of \emph{flipping $L_k$ across $e_k$} to be the triangulated disk 
\[
L_{k+1} = (L_k - (f \cup f')) \cup (g \cup g')
\]

On the other hand, if $L_k$ has no flippable edge then the induction is complete. 

\begin{claim}
\label{Clm:Below}
Suppose that $e'$ is an edge of $\veer(H)$ that is not equal to $e_j$, for any $j < k$.
Then either $e'$ lies in $L_k$ or there is an edge $e$ of $L_k$ which properly west-east spans $e'$.
\end{claim}

\begin{proof}
For $k = 0$ this follows from the definition of $L_0$.
Suppose by induction that the claim holds at stage $k$.
Suppose that $e'$ is an edge of $\veer(H)$ that is not $e_j$ for any $j < k + 1$.
Suppose that no edge $e$ of $L_{k+1}$ properly west-east spans $e'$. 

Suppose that $e'$ lies in $L_{k}$.
Since $e' \neq e_{k}$, we deduce that $e'$ lies in $L_{k+1}$.  
Thus, in this case we are done. 

Suppose instead that $e'$ does not lie in $L_{k}$.
By induction, there is some edge $e$ of $L_{k}$ so that $\cell^{-1}(e)$ properly west-east spans $\cell^{-1}(e')$.
If $e \neq e_{k}$ then $e$ lies in $L_{k+1}$, contrary to assumption. 
Thus $e = e_{k}$ is the only edge of $L_{k}$ whose rectangle properly west-east spans $\cell^{-1}(e')$;
we deduce that the equatorial edges of $t_{k}$ do not give properly spanning rectangles. 
Thus $e'$ is the top edge of $t_{k}$. 
Thus $e'$ lies in $L_{k+1}$, and we are done. 
\end{proof}

\begin{claim}
\label{Clm:Layering}
For every tetrahedron $t'$ of $\veer(H)$, there is an $n$ so that $t' = t_n$.
\end{claim}

\begin{proof}
Let $D(t')$ be the collection of tetrahedra $s$ in $\veer(H)$ which have $\cell^{-1}(s)$ west-east spanning $\cell^{-1}(t')$.
The set $D(t')$ is finite and partially ordered by the spanning relation. 
Note that $t'$ lies in $D(t')$.

Suppose that $\cell^{-1}(e)$ is an edge rectangle in $H$ that west-east spans $\cell^{-1}(t')$.
Let $e'$ be the top edge of $t'$.
Thus $\cell^{-1}(e)$ properly west-east spans $\cell^{-1}(e')$.
By \reflem{TwoEdges}, there is a tetrahedron $t$ so that $e$ is the bottom edge of $t$ and $\cell^{-1}(t)$ lies in the convex hull of the cusps of $\cell^{-1}(e)$ and $\cell^{-1}(e')$.
By \reflem{HullsConvex}, the tetrahedron $t$ lies in $D(t')$.

Recall that $t_k$ is the tetrahedron lying between the triangulated disks $L_k$ and $L_{k+1}$.
Define $D_k(t') = D(t') - \{t_0, t_1, \dots, t_{k-1}\}$.
If $D_k(t')$ is empty then we are done.
Otherwise, suppose that $s' \in D_k(t')$ has a bottom edge which is not flippable in $L_k$. 
Thus some edge $e'$ of $s'$, other than the top edge of $s'$, is not contained in $L_k$.
By \refclm{Below}, there is an edge $e$ of $L_k$ that properly west-east spans $e'$.
By the previous paragraph, $s'$ was not a minimum of $D_k(t')$.
That is, all minima of the partial order on $D_k(t')$ have bottom edges which are flippable in $L_k$.

Any flippable edge in $L_k$, if not removed, remains flippable in $L_{k+1}$.
Thus there is a $k' > k$ so that $D_{k'}(t')$ has fewer elements than $D_k(t')$.
\end{proof}

\begin{claim}
\label{Clm:Thick}
For every non-coastal edge $e$ of $L_0$, there is an $n$ so that $e = e_n$.
\end{claim}

\begin{proof}
This follows from \refcor{BottomEdgeTet} and \refclm{Layering}.
\end{proof}

\refclm{Layering} and \reflem{Hull}\refitm{FiniteContent} imply that the realisation of $\veer(H)$ is a finite collection of finitely triangulated three-balls, perhaps intersecting at separating edges. 
\refclm{Thick} implies that there are no separating edges. 
Thus the realisation $\veer(H)$ is a closed three-ball (minus finitely many boundary points).
This completes the proof of \refprop{ThreeBall}.
\end{proof}

\begin{theorem}
\label{Thm:ThreeSpace}
Suppose that $\calL$ is a loom space.  
Then the realisation of its induced triangulation $\veer(\calL)$ is homeomorphic to $\RR^3$.
\end{theorem}

\begin{proof}
Choose an ordering $(t_i)_{i \in \NN}$ for the tetrahedra of $\veer(\calL)$. 
Applying \refprop{FaceConnected}, we arrange matters so that any initial subsequence of $(t_i)$ is face-connected.
Let $H_n$ be the convex hull of the cusps of the first $n$ tetrahedra.
By \refprop{ThreeBall}, the realisation of $\veer(H_n)$ is a closed three-ball (minus finitely many boundary points). 
Taking interiors, we find that $|\veer(\calL)|$ is an increasing union of open three-balls. 
The theorem now follows from a result of Brown~\cite{Brown61}.
\end{proof}

\renewcommand{\UrlFont}{\tiny\ttfamily}
\renewcommand\hrefdefaultfont{\tiny\ttfamily}
\bibliographystyle{plainurl}
\bibliography{bibfile.bib}
\end{document}